\input{ifpdf.sty}
\pdffalse

\makeatletter

\newdimen\paperwidth
\newdimen\paperheight

\def\papersize#1#2{\let\p@persize\relax\paperwidth#1\paperheight#2}

\def\Afour{\papersize{210truemm}{297truemm}}

\let\p@persize\Afour

\let\onesidestyle\@twosidefalse
\let\twosidestyle\@twosidetrue

\def\margins{\@ifnextchar[{\@margins}{\@margins[\z@]}}

\def\@margins[#1]#2#3{
  \p@persize\dimen0 #3\dimen0 .5\dimen0\normalsize%
  \oddsidemargin-1truein\advance\oddsidemargin#2%
  \evensidemargin-1truein\advance\evensidemargin#2%
  \topmargin-1truein\advance\topmargin\dimen0\headsep\dimen0\footskip\dimen0%
  \textwidth\paperwidth\advance\textwidth-#2\advance\textwidth-#2%
  \textheight\paperheight\advance\textheight-#3\advance\textheight-#3%
  \headheight\baselineskip\advance\topmargin-.5\baselineskip%
  \advance\headsep-.5\baselineskip%
  \footheight\baselineskip
  \advance\textwidth-#1\advance\oddsidemargin#1
  \if@twoside\def\@themargin%
    {\ifodd\count\z@\oddsidemargin\else\evensidemargin\fi}\fi}

\def\headlinesep#1{\advance\topmargin\headsep\advance\topmargin -#1
  \advance\topmargin.5\baselineskip\headsep #1\advance\headsep-.5\baselineskip}

\def\headline{\if@twoside\let\n@xt\h@dlin@\else\let\n@xt\h@@dlin@\fi\n@xt}
  
\def\h@dlin@#1#2{%
  \def\@oddhead{%
    {{\leftskip\z@\rightskip\z@\noindent\normalsize#1}}}
  \def\@evenhead{%
    {{\leftskip\z@\rightskip\z@\noindent\normalsize#2}}}}

\def\h@@dlin@#1{%
  \def\@oddhead{{{\leftskip\z@\rightskip\z@\noindent\normalsize#1}}}}

\def\footline{\if@twoside\let\n@xt\f@tlin@\else\let\n@xt\f@@tlin@\fi\n@xt}

\def\f@tlin@#1#2{%
  \def\@oddfoot{%
    {{\leftskip\z@\rightskip\z@\noindent\normalsize#1}}}
  \def\@evenfoot{%
    {{\leftskip\z@\rightskip\z@\noindent\normalsize#2}}}}

\def\f@@tlin@#1{%
  \def\@oddfoot{{{\leftskip\z@\rightskip\z@\noindent\normalsize#1}}}}

\def\normalpage{\global\@specialpagefalse}

\makeatother

\makeatletter

\def\ft{\@ifnextchar[{\ft@s}{\ft@}}
\def\ft@{\ft@@@s[\f@size]}
\def\ft@s[{\@ifnextchar{a}{\ft@sz[}{\ft@@s[}}
\def\ft@@s[{\@ifnextchar{s}{\ft@sz[}{\ft@@@s[}}
\def\ft@@@s[#1]{\ft@sz[at #1pt]}
\def\ft@sz[#1]#2{\font\fonttemp=#2 #1\fonttemp\ignorespaces}

\def\pbf{\fontseries\bfdefault\selectfont\boldmath}

\makeatother

\makeatletter

\def\@@bold{bold}

\def\widebar{\ifx\math@version\@@bold
  \let\@widebar\@@@widebar\else\let\@widebar\@@widebar\fi\@widebar}

\def\@@widebar#1{\text{\setbox15\hbox{$#1$}%
  \dimen15 0.45\wd15\advance\dimen15 0.15\ht15%
  \dimen16\ht15\advance\dimen16 0.00em\advance\dimen16 0.3ex%
  \dimen17 0.65\wd15\advance\dimen17 0.05\ht15\advance\dimen17 0.1ex%
  \dimen18 0.035em\advance\dimen18 0.00ex
  \put[\dimen15,\dimen16][c]{\vrule depth 0pt height \dimen18 width \dimen17}}#1}

\def\@@@widebar#1{\text{\setbox15\hbox{$#1$}%
  \dimen15 0.45\wd15\advance\dimen15 0.15\ht15%
  \dimen16\ht15\advance\dimen16 0.00em\advance\dimen16 0.26ex%
  \dimen17 0.65\wd15\advance\dimen17 0.05\ht15\advance\dimen17 0.1ex%
  \dimen18 0.05em\advance\dimen18 0.00ex
  \put[\dimen15,\dimen16][c]{\vrule depth 0pt height \dimen18 width \dimen17}}#1}

\makeatother

\def\smallsquare{\raise-.065em\hbox{$\Box$}}
\def\smallblacksquare{%
  \kern.3ex\vrule depth-.03ex height1.27ex width1.15ex \kern-1.45ex \smallsquare}

\def\smallcircc{\mathop{\mkern3.5mu\text{\raise.58ex\hbox{\ft{lcircle10}a}}}}
\def\varemptyset{{\text{\raise.21ex\hbox{$\not$}}\mkern.15mu\mathrm{O}\mkern.15mu}}

  \let\epsilon\varepsilon
\let\textheta\theta      \let\theta\vartheta
          \let\phi\varphi

\documentstyle[12pt]{article}

\makeatletter

\let\Larg@\Large
\let\hug@\huge

\def\usepackage#1{\input{#1.sty}}

\input{geom.sty}
\input{option_keys.sty}

\let\input\@input

\def\r@adlabel#1#2{\global\@namedef{#1@\the\@key}{#2}}

\let\Large\Larg@
\let\huge\hug@

\def\smallskip{\vskip\smallskipamount}
\def\medskip{\vskip\medskipamount}
\def\bigskip{\vskip\bigskipamount}

\def\mytrivlist{\parsep\parskip\@nmbrlistfalse
  \my@trivlist \labelwidth\z@ \leftmargin\z@
  \itemindent\z@ \def\makelabel##1{##1}}

\def\my@trivlist{\global\@newlisttrue \@outerparskip\parskip}

\def\end#1{\csname end#1\endcsname\@checkend{#1}%
  \expandafter\endgroup\if@endpe\@doendpe\fi
  \if@ignore \global\@ignorefalse \ignorespaces\fi}

\input{epsf.sty}
\usepackage{graphicx}

\def\put{\@ifnextchar[{\@put}{\@@rput[\z@,\z@][r]}}
\def\@put[#1]{\@ifnextchar[{\@@put[#1]}{\@@@@@put[#1]}}
\def\@@put[#1][{\@ifnextchar{l}{\@@lput[#1][}{\@@@put[#1][}}
\def\@@@put[#1][{\@ifnextchar{c}{\@@cput[#1][}{\@@@@put[#1][}}
\def\@@@@put[#1][{\@ifnextchar{r}{\@@rput[#1][}{\relax}}
\def\@@@@@put[{\@ifnextchar{l}{\@@lput[\z@,\z@][}{\@@@@@@put[}}
\def\@@@@@@put[{\@ifnextchar{c}{\@@cput[\z@,\z@][}{\@@@@@@@put[}}
\def\@@@@@@@put[{\@ifnextchar{r}{\@@rput[\z@,\z@][}{\@@@@@@@@put[}}
\def\@@@@@@@@put[#1]{\@@rput[#1][r]}

\let\hm@d@\leavevmode

\long\def\@@lput[#1,#2][l]#3{\setbox0\hbox{#3}\hm@d@\raise#2\hbox to\z@{\dimen0 #1%
  \advance\dimen0-\wd0\kern\dimen0\dp0\z@\ht0\z@\wd0\z@\box0\hss}\ignorespaces}
\long\def\@@cput[#1,#2][c]#3{\setbox0\hbox{#3}\hm@d@\raise#2\hbox to\z@{\dimen0 #1%
  \advance\dimen0-.5\wd0\kern\dimen0\dp0\z@\ht0\z@\wd0\z@\box0\hss}\ignorespaces}
\long\def\@@rput[#1,#2][r]#3{\setbox0\hbox{\kern#1\raise#2\hbox{#3}}%
  \dp0\z@\ht0\z@\wd0\z@\hm@d@\box0\ignorespaces}

\def\flbox{\@ifnextchar[{\@flbox}{\@@rflbox[\z@,\z@][r]}}
\def\@flbox[#1]{\@ifnextchar[{\@@flbox[#1]}{\@@@@@flbox[#1]}}
\def\@@flbox[#1][{\@ifnextchar{l}{\@@lflbox[#1][}{\@@@flbox[#1][}}
\def\@@@flbox[#1][{\@ifnextchar{c}{\@@cflbox[#1][}{\@@@@flbox[#1][}}
\def\@@@@flbox[#1][{\@ifnextchar{r}{\@@rflbox[#1][}{\relax}}
\def\@@@@@flbox[{\@ifnextchar{l}{\@@lflbox[\z@,\z@][}{\@@@@@@flbox[}}
\def\@@@@@@flbox[{\@ifnextchar{c}{\@@cflbox[\z@,\z@][}{\@@@@@@@flbox[}}
\def\@@@@@@@flbox[{\@ifnextchar{r}{\@@rflbox[\z@,\z@][}{\@@@@@@@@flbox[}}
\def\@@@@@@@@flbox[#1]{\@@rflbox[#1][r]}
\long\def\@@lflbox[#1,#2][l]#3{\@@lput[#1,#2][l]{%
  \vtop{\leftskip\z@\parindent\z@\raggedleft\hm@d@#3}}}
\long\def\@@cflbox[#1,#2][c]#3{\@@cput[#1,#2][c]{%
  \vtop{\leftskip\z@\parindent\z@\raggedcenter\hm@d@#3}}}
\long\def\@@rflbox[#1,#2][r]#3{\@@rput[#1,#2][r]{%
  \vtop{\leftskip\z@\parindent\z@\raggedright\hm@d@#3}}}

\def\maketitle{\par
 \begingroup
 \def\thefootnote{\fnsymbol{footnote}}
 \def\@makefnmark{\hbox 
 to 0pt{$^{\@thefnmark}$\hss}} 
 \if@twocolumn 
 \twocolumn[\@maketitle] 
 \else 
 \global\@topnum\z@ \@maketitle \fi\thispagestyle{plain}\@thanks
 \endgroup
 \setcounter{footnote}{0}
 \let\maketitle\relax
 \let\@maketitle\relax
 \gdef\@thanks{}\gdef\@author{}\gdef\@title{}\let\thanks\relax}

\def\@maketitle{ 
 \null
 \vskip 2em \begin{center}
 {\LARGE \@title \par} \vskip 1.5em {\large \lineskip .5em
\begin{tabular}[t]{c}\@author 
 \end{tabular}\par} 
 \vskip 1em {\large \@date} \end{center}
 \par
 \vskip 1.5em}
 
\def\partbeforeskip#1{\def\p@rtbeforeskip{#1}}
\def\partstyle#1{\def\p@rtstyl@{#1}}
\def\partdot#1{\def\partd@t{#1}}
\def\partafterskip#1{\def\p@rtafterskip{#1}}
\def\partintrostyle#1{\def\partintr@styl@{#1}}
\def\partintrodot#1{\def\partintr@dot{#1}}
\long\def\partintrosep#1{\long\def\partintr@sep{#1}}
\def\partnewpagetrue{\def\p@rtnewp@ge{\newpage}}
\def\partnewpagefalse{\long\def\p@rtnewp@ge{\par}}

\partbeforeskip{4ex}
\partstyle{\centering\Large\bf}
\partdot{}
\partafterskip{3ex}
\partintrostyle{\large}
\partintrodot{}
\partintrosep{\par}
\partnewpagefalse

\def\partname{Part}
\def\part{\p@rtnewp@ge\addvspace\p@rtbeforeskip\@afterindentfalse\secdef\@part\@spart}

\def\@part[#1]#2{\ifnum \c@secnumdepth >-1\relax  
        \refstepcounter{part}                     
        \def\@tempa{\addcontentsline{toc}{part}}  %
        \expandafter\@tempa\expandafter{\thepart  
          \hspace{1em}#1}\else                    
        \addcontentsline{toc}{part}{#1}\fi        
   {\p@rtstyl@                       
    \ifnum \c@secnumdepth >-1\relax        
      {\partintr@styl@\partname\ \thepart  
       \partintr@dot}\partintr@sep\nobreak 
    \fi                                    
    #2\partd@t\markboth{}{}\par}
    \nobreak                       
    \vskip\p@rtafterskip           
   \@afterheading                  
    }                              

\def\@spart#1{{\p@rtcentering\p@rtstyl@                      
    #1\partd@t\par}                 
    \nobreak                        
    \vskip\p@rtafterskip            
    \@afterheading                  
  }                                 

\newif\ifsection@ftind
\newif\ifsection@ftpar

\def\sectionbeforeskip#1{\def\s@ctbeforeskip{#1}}
\def\sectionstyle#1{\def\s@ctstyl@{#1}}
\def\sectiondot#1{\def\sectiond@t{#1}}
\def\sectionafterskip#1{\def\s@ctafterskip{#1}}
\def\sectionintrostyle#1{\def\sectionintr@styl@{#1}}
\def\sectionintro#1{\def\sectionintr@{#1}}
\def\sectionintrodot#1{\def\sectionintr@dot{#1}}
\def\sectionintrosep#1{\def\sectionintr@sep{#1}}
\def\sectionindenttrue{\def\s@ctind{\parindent}}
\def\sectionindentfalse{\def\s@ctind{\z@}}
\def\sectionafterindenttrue{\section@ftindtrue}
\def\sectionafterindentfalse{\section@ftindfalse}
\def\sectionafternewlinetrue{\section@ftpartrue}
\def\sectionafternewlinefalse{\section@ftparfalse}

\newif\ifsubsection@ftind
\newif\ifsubsection@ftpar

\def\subsectionbeforeskip#1{\def\ss@ctbeforeskip{#1}}
\def\subsectionstyle#1{\def\ss@ctstyl@{#1}}
\def\subsectiondot#1{\def\subsectiond@t{#1}}
\def\subsectionafterskip#1{\def\ss@ctafterskip{#1}}
\def\subsectionintrostyle#1{\def\subsectionintr@styl@{#1}}
\def\subsectionintro#1{\def\subsectionintr@{#1}}
\def\subsectionintrodot#1{\def\subsectionintr@dot{#1}}
\def\subsectionintrosep#1{\def\subsectionintr@sep{#1}}
\def\subsectionindenttrue{\def\ss@ctind{\parindent}}
\def\subsectionindentfalse{\def\ss@ctind{\z@}}
\def\subsectionafterindenttrue{\subsection@ftindtrue}
\def\subsectionafterindentfalse{\subsection@ftindfalse}
\def\subsectionafternewlinetrue{\subsection@ftpartrue}
\def\subsectionafternewlinefalse{\subsection@ftparfalse}

\newif\ifsubsubsection@ftind
\newif\ifsubsubsection@ftpar

\def\subsubsectionbeforeskip#1{\def\sss@ctbeforeskip{#1}}
\def\subsubsectionstyle#1{\def\sss@ctstyl@{#1}}
\def\subsubsectiondot#1{\def\subsubsectiond@t{#1}}
\def\subsubsectionafterskip#1{\def\sss@ctafterskip{#1}}
\def\subsubsectionintrostyle#1{\def\subsubsectionintr@styl@{#1}}
\def\subsubsectionintro#1{\def\subsubsectionintr@{#1}}
\def\subsubsectionintrodot#1{\def\subsubsectionintr@dot{#1}}
\def\subsubsectionintrosep#1{\def\subsubsectionintr@sep{#1}}
\def\subsubsectionindenttrue{\def\sss@ctind{\parindent}}
\def\subsubsectionindentfalse{\def\sss@ctind{\z@}}
\def\subsubsectionafterindenttrue{\subsubsection@ftindtrue}
\def\subsubsectionafterindentfalse{\subsubsection@ftindfalse}
\def\subsubsectionafternewlinetrue{\subsubsection@ftpartrue}
\def\subsubsectionafternewlinefalse{\subsubsection@ftparfalse}

\newif\ifparagraph@ftind
\newif\ifparagraph@ftpar

\def\paragraphbeforeskip#1{\def\p@rbeforeskip{#1}}
\def\paragraphstyle#1{\def\p@rstyl@{#1}}
\def\paragraphdot#1{\def\paragraphd@t{#1}}
\def\paragraphafterskip#1{\def\p@rafterskip{#1}}
\def\paragraphintrostyle#1{\def\paragraphintr@styl@{#1}}
\def\paragraphintro#1{\def\paragraphintr@{#1}}
\def\paragraphintrodot#1{\def\paragraphintr@dot{#1}}
\def\paragraphintrosep#1{\def\paragraphintr@sep{#1}}
\def\paragraphindenttrue{\def\p@rind{\parindent}}
\def\paragraphindentfalse{\def\p@rind{\z@}}
\def\paragraphafterindenttrue{\paragraph@ftindtrue}
\def\paragraphafterindentfalse{\paragraph@ftindfalse}
\def\paragraphafternewlinetrue{\paragraph@ftpartrue}
\def\paragraphafternewlinefalse{\paragraph@ftparfalse}

\newif\ifsubparagraph@ftind
\newif\ifsubparagraph@ftpar

\def\subparagraphbeforeskip#1{\def\sp@rbeforeskip{#1}}
\def\subparagraphstyle#1{\def\sp@rstyl@{#1}}
\def\subparagraphdot#1{\def\subparagraphd@t{#1}}
\def\subparagraphafterskip#1{\def\sp@rafterskip{#1}}
\def\subparagraphintrostyle#1{\def\subparagraphintr@styl@{#1}}
\def\subparagraphintro#1{\def\subparagraphintr@{#1}}
\def\subparagraphintrodot#1{\def\subparagraphintr@dot{#1}}
\def\subparagraphintrosep#1{\def\subparagraphintr@sep{#1}}
\def\subparagraphindenttrue{\def\sp@rind{\parindent}}
\def\subparagraphindentfalse{\def\sp@rind{\z@}}
\def\subparagraphafterindenttrue{\subparagraph@ftindtrue}
\def\subparagraphafterindentfalse{\subparagraph@ftindfalse}
\def\subparagraphafternewlinetrue{\subparagraph@ftpartrue}
\def\subparagraphafternewlinefalse{\subparagraph@ftparfalse}

\sectionbeforeskip{\bigskipamount}
\sectionstyle{\large\bf}
\sectiondot{}
\sectionafterskip{.5\bigskipamount}
\sectionintrostyle{}
\sectionintro{}
\sectionintrodot{.}
\sectionintrosep{1.25ex}
\sectionindentfalse
\sectionafterindenttrue
\sectionafternewlinetrue

\subsectionbeforeskip{.8\bigskipamount}
\subsectionstyle{\normalsize\bf}
\subsectiondot{}
\subsectionafterskip{.4\bigskipamount}
\subsectionintrostyle{}
\subsectionintro{}
\subsectionintrodot{.}
\subsectionintrosep{1.25ex}
\subsectionindentfalse
\subsectionafterindenttrue
\subsectionafternewlinetrue

\subsubsectionbeforeskip{.6\bigskipamount}
\subsubsectionstyle{\normalsize\bf}
\subsubsectiondot{}
\subsubsectionafterskip{.3\bigskipamount}
\subsubsectionintrostyle{}
\subsubsectionintro{}
\subsubsectionintrodot{.}
\subsubsectionintrosep{1.25ex}
\subsubsectionindentfalse
\subsubsectionafterindenttrue
\subsubsectionafternewlinetrue

\paragraphbeforeskip{.5\bigskipamount}
\paragraphstyle{\normalsize\bf}
\paragraphdot{.}
\paragraphafterskip{1.25ex}
\paragraphintrostyle{}
\paragraphintro{}
\paragraphintrodot{.}
\paragraphintrosep{1.25ex}
\paragraphindentfalse
\paragraphafterindenttrue
\paragraphafternewlinefalse

\subparagraphbeforeskip{.5\bigskipamount}
\subparagraphstyle{\normalsize\bf}
\subparagraphdot{.}
\subparagraphafterskip{1.25ex}
\subparagraphintrostyle{}
\subparagraphintro{}
\subparagraphintrodot{.}
\subparagraphintrosep{1.25ex}
\subparagraphindenttrue
\subparagraphafterindenttrue
\subparagraphafternewlinefalse

\let\@partoken\par
\long\def\@@gobble#1{}
\def\ignorepar{\@ifnextchar\@partoken{\expandafter\ignorepar\@@gobble}{\ignorespaces}}

\def\@startsection#1#2#3#4#5#6{
   \@tempskipa #4\relax
   \csname if#1@ftind\endcsname\@afterindenttrue\else\@afterindentfalse\fi
   \advance\@tempskipa by\presection
   \if@nobreak \everypar{}\else
     \addpenalty{\@secpenalty}\addvspace{\@tempskipa}%
     \allowbreak\vskip -\presection \fi \@ifstar
     {\@ssect{#1}{#2}{#3}{#4}{#5}{#6}}{\@dblarg{\@sect{#1}{#2}{#3}{#4}{#5}{#6}}}}

\def\@sect#1#2#3#4#5#6[#7]#8{\def\object@type{#1}%
   \ifnum #2>\c@secnumdepth\def\@svsec{}\def\@tempb{}%
      \else\refstepcounter{#1}\def\@svsec{{\csname #1intr@styl@\endcsname%
        {\csname #1intr@\endcsname}\csname the#1\endcsname%
        \csname #1intr@dot\endcsname\kern\csname #1intr@sep\endcsname}}%
        \edef\@tempb{\noexpand\numberline{\csname the#1\endcsname}}\fi%
   \def\@tempa{\addcontentsline{toc}{#1}}%
   \csname if#1@ftpar\endcsname%
      \begingroup #6\relax%
        \@hangfrom{\hskip #3\relax\@svsec}{\interlinepenalty \@M{#8}%
        \csname #1d@t\endcsname\par}%
      \endgroup%
      \csname #1mark\endcsname{#7}%
      \expandafter\@tempa\expandafter{\@tempb #7}%
      \ifautolabel\label*{#8}\fi%
   \else%
      \def\@svsechd{#6\hskip #3\relax%
         \@svsec{#8}%
         \csname #1d@t\endcsname%
         \csname #1mark\endcsname{#7}%
         \expandafter\@tempa\expandafter{\@tempb #7}%
         \ifautolabel\label*{#8}\fi}\fi%
   \@xsect{#1}{#5}\ignorepar}

\def\@ssect#1#2#3#4#5#6#7{%
   \ifnum #2>\c@secnumdepth\def\@tempb{}\else \def\@tempb{\numberline{}}\fi%
     \def\@tempa{\addcontentsline{toc}{s#1}}%
     \csname if#1@ftpar\endcsname
        \begingroup #6\relax
           \@hangfrom{\hskip #3}{\interlinepenalty \@M{#7}%
           \csname #1d@t\endcsname\par}%
        \endgroup
        \csname s#1mark\endcsname{#7}%
        \ifstarredcontents\expandafter\@tempa\expandafter{\@tempb #7}\fi%
        \ifautolabel\label*{#7}\fi%
     \else%
        \def\@svsechd{#6\hskip #3\relax{#7}%
        \csname #1d@t\endcsname%
        \csname s#1mark\endcsname{#7}%
        \ifautolabel\label*{#7}\fi}\fi
   \@xsect{#1}{#5}\ignorepar}

\def\@xsect#1#2{
   \csname if#1@ftpar\endcsname 
       \par \nobreak \vskip #2\relax \@afterheading
    \else \global\@nobreakfalse \global\@noskipsectrue
       \everypar{\if@noskipsec \global\@noskipsecfalse
                   \clubpenalty\@M \hskip -\parindent
                   \begingroup \@svsechd \endgroup \unskip
                   \hskip #2\relax  
                  \else \clubpenalty \@clubpenalty
                    \everypar{}\fi}\fi\ignorespaces}

\def\section{\@startsection{section}{1}{\s@ctind}
  {\s@ctbeforeskip}{\s@ctafterskip}{\s@ctstyl@}}
\def\subsection{\@startsection{subsection}{2}{\ss@ctind}
  {\ss@ctbeforeskip}{\ss@ctafterskip}{\ss@ctstyl@}}
\def\subsubsection{\@startsection{subsubsection}{3}{\sss@ctind}
  {\sss@ctbeforeskip}{\sss@ctafterskip}{\sss@ctstyl@}}
\def\paragraph{\@startsection{paragraph}{4}{\p@rind}
  {\p@rbeforeskip}{\p@rafterskip}{\p@rstyl@}}
\def\subparagraph{\@startsection{subparagraph}{4}{\sp@rind}
  {\sp@rbeforeskip}{\sp@rafterskip}{\sp@rstyl@}}

\def\statementabove#1{\def\th@bove{#1}}
\def\statementstyle#1{\def\thstyl@{#1}}
\def\statementbelow#1{\def\thb@low{#1}}
\def\statementindentfalse{\let\thind@nt\relax}
\def\statementindenttrue{\let\thind@nt\indent}

\def\statementintrostyle#1{\def\thintr@style{#1}}
\def\statementintrodot#1{\def\thintr@dot{#1}}
\def\statementintrosep#1{\def\thintr@sep{#1}}
\def\statementintrobrackets#1#2{\def\thintr@left{#1}\def\thintr@right{#2}}

\statementabove{\medskip}
\statementstyle{\sl}
\statementbelow{\medskip}
\statementindenttrue

\statementintrostyle{\normalshape\bf}
\statementintrodot{.}
\statementintrosep{\kern1.25ex}
\statementintrobrackets{(}{)}

\def\@thskip{\dimen100\lastskip\vskip-\dimen100%
  \th@bove\dimen101\lastskip\vskip-\dimen101%
  \ifdim\dimen100>\dimen101\else\dimen100\dimen101\fi\vskip\dimen100\vskip0pt}

\long\def\@@newtheorem#1#2#3{%
  \newenvironment{#3}%
    {\def\object@type{#3}\par\@thskip%
     \@ifnextchar[{\@enva{#3}{\thstyl@#1{#2}}}{\@envb{#3}{\thstyl@#1{#2}}}}%
    {\end{#3@}}%
  \@ifnextchar[{\@nothm{#3}}{\@nnthm{#3}}}

\def\@nothm#1[#2]#3{%
  \@ifundefined{c@#2}{\@latexerr{No theorem environment `#2' defined}\@eha}%
  {\expandafter\@ifdefinable\csname #1@\endcsname
  {\global\@namedef{the#1}{\@nameuse{the#2}}%
   \global\@namedef{c@#1}{\@nameuse{c@#2}}
   \global\@namedef{p@#1}{\@nameuse{p@#2}}
   \global\@namedef{#1@}{\@nnnthm{#2}{#3}}%
   \global\@namedef{end#1@}{\@endtheorem}}}}

\def\@nnnthm#1#2{\refstepcounter
    {#1}\@ifnextchar[{\@ynnnthm{#1}{#2}}{\@xnnnthm{#1}{#2}}}
 
\def\@xnnnthm#1#2{\@begintheorem{#2}{\csname the#1\endcsname}\ignorespaces}
\def\@ynnnthm#1#2[#3]{\@opargbegintheorem{#2}{\csname the#1\endcsname}{#3}\ignorespaces}

\def\renewtheorem{\@ifnextchar[{\@renewtheorem}{\@renewtheorem[{}{}]}}

\long\def\@renewtheorem[#1]{\@@renewtheorem#1}

\long\def\@@renewtheorem#1#2#3{%
  \expandafter\let\csname#3@\endcsname\undefined
  \renewenvironment{#3}%
    {\def\object@type{#3}\par\@thskip%
     \@ifnextchar[{\@enva{#3}{\thstyl@#1{#2}}}{\@envb{#3}{\thstyl@#1{#2}}}}%
    {\end{#3@}}%
  \@ifnextchar[{\@nothm{#3}}{\@nnthm{#3}}}

\def\@begintheorem#1#2{\@opargbegintheorem{#1}{#2}{}}

\def\@opargbegintheorem#1#2#3{%
        \edef\@tempx{#1}%
        \expandafter\let\expandafter\@tempy#2
        \def\@tempz{#3}%
        \mytrivlist\item[\thind@nt\hskip\labelsep%
        {\thintr@style%
          #1\ifx\@tempx\@empty\else\ifx\@tempy\relax\else\kern1ex\fi\fi#2%
          \ifx\@tempz\@empty%
            \ifx\@tempx\@empty\ifx\@tempy\relax%
            \else\thintr@dot\thintr@sep\fi\else\thintr@dot\thintr@sep\fi%
            \else%
            \ifx\@tempx\@empty\ifx\@tempy\relax\else\kern1ex\fi\else\kern1ex\fi%
           \thintr@left{#3}\thintr@right\thintr@dot\thintr@sep\fi}%
            \hskip-\labelsep]%
        \ifautolabel\label*{#3}\fi}

\def\@endtheorem{\endtrivlist\thb@low}

\def\proofname{Proof}

\def\proofabove#1{\def\pf@bove{#1}}
\def\proofstyle#1{\def\pfstyl@{#1}}
\def\proofbelow#1{\def\pfb@low{#1}}
\def\proofindentfalse{\let\pfind@nt\relax}
\def\proofindenttrue{\let\pfind@nt\indent}

\def\proofintrostyle#1{\def\pfintr@style{#1}}
\def\proofintrodot#1{\def\pfintr@dot{#1}}
\def\proofintrosep#1{\def\pfintr@sep{#1}}
\def\proofintrobrackets#1#2{\def\pfintr@left{#1}\def\pfintr@right{#2}}

\proofabove{\medskip}
\proofstyle{}
\proofbelow{\medskip}
\proofindenttrue

\proofintrostyle{\sl}
\proofintrodot{.}
\proofintrosep{\kern1.25ex}
\proofintrobrackets{of\kern1ex}{}

\def\@pfskip{\dimen100\lastskip\vskip-\dimen100%
  \pf@bove\dimen101\lastskip\vskip-\dimen101%
  \ifdim\dimen100>\dimen101\else\dimen100\dimen101\fi\vskip\dimen100\vskip0pt}

\renewenvironment{proof}%
  {\@pfskip\mytrivlist\item[\pfind@nt]\@ifnextchar[{\pro@f}{\pro@f[\prooftag]}}
  {\ifvoid\provedbox\else\hproved\fi\endtrivlist\pfb@low}

\def\pro@f[#1]{\setbox\provedbox\hbox{\provedboxcontents{#1}}\proofintro{#1}}

\def\proofintro#1{\expandafter\def\expandafter\@tempa\expandafter{#1}%
  {\pfintr@style{\proofname\ifx\@tempa\empty\else\kern1ex\pfintr@left{#1}%
  \pfintr@right\fi}\pfintr@dot\pfintr@sep}\pfstyl@\ignorespaces}

\def\provedmark#1{\def\prm@rk{#1}}
\def\provedsep#1{\def\prs@p{#1}}

\provedmark{$\square$}
\provedsep{\kern1.25ex}

\def\provedtexttrue{\def\prb@x##1{\fbox{\small##1}}}
\def\provedtextfalse{\def\prb@x##1{\prm@rk}}
\def\provedmarkrighttrue{\let\prhf@l\hfill}
\def\provedmarkrightfalse{\let\prhf@l\relax}

\provedtextfalse
\provedmarkrightfalse

\def\provedboxcontents#1{\expandafter\def\expandafter\@tempa\expandafter{#1}%
  \ifx\@tempa\empty\prm@rk\else\prb@x{#1}\fi}

\def\proved{\ifmmode\eqno{\box\provedbox}\else\hproved\fi}

\def\hproved{\unskip\nobreak\prhf@l\penalty50\prs@p\hbox{}\nobreak\prhf@l
  \box\provedbox{\finalhyphendemerits=0\par}}

\def\captionstyle#1{\def\c@ptstyl@{#1}}
\def\captionintrostyle#1{\def\c@pintr@style{#1}}
\def\captionintrodot#1{\def\c@pintr@dot{#1}}
\def\captionintrosep#1{\def\c@pintr@sep{#1}}

\captionstyle{\small\sf}
\captionintrostyle{\bf}
\captionintrodot{.}
\captionintrosep{\hskip1.25ex}

\long\def\@makecaption#1#2{%
    \vskip\captionskip
    \setbox\@tempboxa\hbox{%
      \ifproofing\@ifundefined{the@label}{}
        {\hbox to 0pt{\vbox to 0pt{\vss\hbox{\tiny\the@label}\bigskip}\hss}}\fi
      \c@ptstyl@{\c@pintr@style #1\c@pintr@dot}\ignorespaces #2}%
    \@captionwidth=\hsize \advance\@captionwidth-2\@captionmargin
    \ifdim \wd\@tempboxa >\@captionwidth {%
        \rightskip=\@captionmargin\leftskip=\@captionmargin
        \unhbox\@tempboxa\par}%
      \else
        \hbox to\hsize{\hfil\box\@tempboxa\hfil}%
    \fi}

\def\end@Float#1{%
  \expandafter\caption\expandafter[\the@title]{%
   {\c@pintr@style%
   \ifx\the@caption\empty\ifx\the@title\empty
   \else\c@pintr@sep\fi\else\c@pintr@sep\fi
    \the@title\ifx\the@caption\empty%
     \expandafter\label\expandafter*\expandafter{\the@label}%
    \else\ifx\the@title\empty%
     \expandafter\label\expandafter*\expandafter{\the@label}%
    \else\c@pintr@dot\c@pintr@sep%
     \expandafter\label\expandafter*\expandafter{\the@label}\fi\fi}%
   \ignorespaces\the@caption}%
  \end{#1}}

\renewenvironment{Figure}{\@ifnextchar[%
  {\@myFloat{figure}}{\@myFloat{figure}[htbp]}}{\end@Float{figure}}

\def\@myFloat#1[#2]#3{%
  \def\color@hbox{}\def\color@vbox{}\def\color@endbox{}%
  \begin{#1}[#2]\def\the@label{#3}}

\def\fig#1{\@ifnextchar[{\@fig{#1}}{\@fig{#1}[0pt]}}

\def\@fig#1[#2]#3{\@ifnextchar[{\@@fig{#1}[#2]{#3}}{\@@fig{#1}[#2]{#3}[0pt]}}
\def\@@fig#1[#2]#3[#4]#5#6{%
  \def\the@title{#5}\def\the@caption{#6}\centerline{\fig@{#1}{#2}{#3}}\vskip#4}

\def\fig@@#1#2#3{\leavevmode{\figstyl@\vrule width 0pt height 1.8ex%
 \smash{\framebox{\strut\def\@temp{#1}\ifx\@temp\@empty{ #3 }\else{ #1 }\fi}}}}
\def\fig@@@#1#2#3{\leavevmode\kern#2\epsfbox{#3}}

\def\figstyle#1{\def\figstyl@{#1}}

\figstyle{\fontfamily{cmr}\fontshape{n}\fontseries{m}\selectfont\normalsize}

\newcounter{diagram}

\let\thediagram\theequation
\def\ftype@diagram{2}
\def\ext@diagram{lod}
\def\diagram{\@float{diagram}}
\let\enddiagram\end@float
\newif\if@diagnum

\def\diag#1{\@ifnextchar[{\@diag{#1}}{\@diag{#1}[0pt]}}

\def\@diag#1[#2]#3{\@ifnextchar[{\@@diag{#1}[#2]{#3}}{\@@diag{#1}[#2]{#3}[0pt]}}

\def\@@diag#1[#2]#3[#4]#5{
  \def\the@tag{#5}\@eqnswtrue%
  \centerline{\setbox0\hbox{\diag@{#1}{#2}{#3}}
  \dimen0 -0.5\wd0\dimen1 0.5\ht0\box0%
  \advance\dimen0 0.5\hsize\advance\dimen0 -\rightskip\advance\dimen1 #4%
  \let\@currentlabel\the@tag%
  \setbox0\hbox to 0pt{\hss%
    \fontfamily{cmr}\fontshape{n}\fontseries{m}\selectfont(\the@tag)}%
  \ifx\the@tag\@empty\refstepcounter{equation}\let\@currentlabel\theequation%
    \setbox0\hbox to 0pt{\hss%
      \fontfamily{cmr}\fontshape{n}\fontseries{m}\selectfont(\thediagram)}\fi%
  \if@eqnsw\else\let\@currentlabel\relax\setbox0\hbox to 0pt{}\fi%
  \advance\dimen1 -0.5\ht0%
  \put[\dimen0,\dimen1][l]{%
    \box0\expandafter\label\expandafter*\expandafter{\the@label}\kern0.15em}}}

\def\diag@@#1#2#3{\leavevmode{\diagstyl@\vrule width 0pt height 1.8ex%
 \smash{\framebox{\strut\def\@temp{#1}\ifx\@temp\@empty{ #3 }\else{ #1 }\fi}}}}
\def\diag@@@#1#2#3{\leavevmode\kern#2\epsfbox{#3}}

\def\diagstyle#1{\def\diagstyl@{#1}}

\diagstyle{\fontfamily{cmr}\fontshape{n}\fontseries{m}\selectfont\normalsize}

\def\showfiguresfalse{\let\fig@\fig@@}
\def\showfigurestrue{\let\fig@\fig@@@}

\def\showdiagramsfalse{\let\diag@\diag@@}
\def\showdiagramstrue{\let\diag@\diag@@@}

\showfigurestrue
\showdiagramstrue

\def\n@number{\@eqnswfalse\let\@currentlabel\relax\let\the@tag\relax}

\def\equation{$$
  \@eqnswtrue\def\object@type{equation}\let\nonumber\n@number%
  \advance\c@equation1\edef\@currentlabel{\theequation}\advance\c@equation-1%
  \def\the@tag{\refstepcounter{equation}\eqno\hbox{\@eqnnum}}}

\def\tag#1{\edef\@currentlabel{#1}\def\the@tag{\eqno\hbox{\reset@font\rm(#1)}}}

\def\endequation{\the@tag$$
  \global\@ignoretrue}

\let\it@m\item
\def\item{\@ifnextchar[{\item@}{\item@@}}
\def\item@[#1]{\it@m[#1]\vskip-\lastskip\vskip\itemsep}
\def\item@@{\it@m\vskip-\lastskip\vskip\itemsep}
\def\s@titemsep{\@ifnextchar[{\s@@titemsep}{\relax}}
\def\s@@titemsep[#1]{\itemsep#1}

\let\@itemize\itemize
\let\@enditemize\enditemize
\renewenvironment{itemize}
{\@itemize\itemsep3pt\parsep0pt\topsep0pt\partopsep0pt\s@titemsep}
{\@enditemize\vskip-\lastskip\vskip\itemsep}

\let\@enumerate\enumerate
\let\@endenumerate\endenumerate
\renewenvironment{enumerate}
{\@enumerate\itemsep3pt\parsep0pt\topsep0pt\partopsep0pt\s@titemsep}
{\@endenumerate\vskip-\lastskip\vskip\itemsep}

\let\@description\description
\let\@enddescription\enddescription
\renewenvironment{description}
{\@description\itemsep3pt\parsep0pt\topsep0pt\partopsep0pt\s@titemsep}
{\@enddescription\vskip-\lastskip\vskip\itemsep}

\topsep\dimen200
\partopsep\dimen201

\@definecounter{bibenumi}

\def\thebibliography#1{%
 \section*{\refname}\vskip-\lastskip%
 \list{[\arabic{bibenumi}]}{\topsep0pt\settowidth\labelwidth{[#1]}%
 \leftmargin\labelwidth\advance\leftmargin\labelsep\usecounter{bibenumi}}%
 \def\newblock{\hskip .11em plus .33em minus .07em}%
 \sloppy\clubpenalty4000\widowpenalty4000\sfcode`\.=1000\relax}

\let\@ref@\ref
\let\@pageref@\pageref
\let\@fullref@\fullref
\let\@Fullref@\Fullref
\let\@reftype@\reftype
\let\@Reftype@\Reftype
\let\@label@\label
\let\@cite@\cite
\let\@bibitem@\bibitem

\def\label{\@ifnextchar*{\label@}{\label@{}}}
\def\label@#1#2{\@label@#1{#2}\putl@bel{#2}\ignorespaces}
\def\putl@b@l#1{\put[0pt,.25\baselineskip]{%
  \hbox{\labc@lor{\fontfamily{cmr}\fontshape{n}\fontseries{m}\selectfont%
  \tiny\setbox5\hbox{\vphantom{X}\smash{\ns#1}}%
  \hbox to 0pt{\hss\tiny$\blacktriangledown$\kern-.085em}%
  \raise2.25ex\hbox to 0pt{\hss\framebox{\box5}}}}}}

\def\putr@fl@bel#1{{\let\labc@lor\refc@lor\putl@bel{#1}}}
\def\ref@#1{\@ref@{#1}\putr@fl@bel{#1}}
\def\pageref@#1{\@pageref@{#1}\putr@fl@bel{#1}}
\def\fullref@#1{\@fullref@{#1}\putr@fl@bel{#1}}
\def\Fullref@#1{\@Fullref@{#1}\putr@fl@bel{#1}}
\def\reftype@#1{\@reftype@{#1}\putr@fl@bel{#1}}
\def\Reftype@#1{\@Reftype@{#1}\putr@fl@bel{#1}}

\def\ref@@#1{\leavevmode\refc@lor{\rm$\langle$#1$\rangle$}}
\let\pageref@@\ref@@
\let\Fullref@@\ref@@
\let\fullref@@\ref@@
\let\reftype@@\ref@@
\let\Reftype@@\ref@@

\def\bibitem{\@ifnextchar[{\bibitem@@}{\bibitem@@@}}
\def\bibitem@@[#1]#2{\@bibitem@[#1]{#2}\putl@bel{#2}}
\def\bibitem@@@#1{\@bibitem@{#1}\putl@bel{#1}\ignorespaces}

\def\cit@{\@ifnextchar[{\@cit@@@}{\@cit@@}}
\def\@cit@@#1{\@cite@{#1}{\let\labc@lor\citc@lor\putl@bel{#1}}}
\def\@cit@@@[#1]#2{\@cite@[#1]{#2}{\let\labc@lor\citc@lor\putl@bel{#2}}}
\def\cit@@{\@ifnextchar[{\cit@@@@}{\cit@@@}}
\def\cit@@@#1{\leavevmode{\citc@lor\rm[#1]}}
\def\cit@@@@[#1]#2{\leavevmode{\citc@lor\rm[#2, #1]}}

\def\showcitationstrue{\let\cite\cit@}
\def\showcitationsfalse{\let\cite\cit@@}

\def\showreferencestrue{%
  \let\ref\ref@\let\pageref\pageref@%
  \let\fullref\fullref@\let\Fullref\Fullref@%
  \let\reftype\reftype@\let\Reftype\Reftype@}
\def\showreferencesfalse{%
  \let\ref\ref@@\let\pageref\pageref@@%
  \let\fullref\fullref@@\let\Fullref\Fullref@@%
  \let\reftype\reftype@@\let\Reftype\Reftype@@}

\def\showlabelstrue{\let\putl@bel\putl@b@l}
\def\showlabelsfalse{\let\putl@bel\hid@@}

\def\postit@{\@ifnextchar[{\postit@@}{\p@tp@stit}}
\def\postit@@[#1]{\postit@@@#1,@}
\def\postit@@@#1,{\@ifnextchar{@}{\p@@tp@stit{#1}}{\postit@@@@#1,}}
\def\postit@@@@#1,#2,@{\p@@@tp@stit{#1}{#2}}

\long\def\p@tp@stit#1{\put[0pt,.25\baselineskip]{%
  \hbox{\postitc@lor{\fontfamily{cmr}\fontshape{n}\fontseries{m}\selectfont%
  \tiny\setbox5\hbox{\vphantom{X}\smash{\ns#1}}%
  \hbox to 0pt{\hss\tiny$\blacktriangledown$\kern-.085em}%
  \raise2.25ex\hbox to 0pt{\hss\framebox{\box5}}}}}}

\long\def\p@@tp@stit#1@#2{\put[0pt,.25\baselineskip]{%
  \hbox{\postitc@lor{\fontfamily{cmr}\fontshape{n}\fontseries{m}\selectfont%
  \tiny\setbox5\hbox{\vbox{\hsize#1\leftskip\z@\raggedright
  \parindent\z@{\ns#2\par}\vss}}%
  \hbox to 0pt{\hss\tiny$\blacktriangledown$\kern-.085em}%
  \raise2.25ex\hbox to 0pt{\hss\framebox{\box5}}}}}}

\long\def\p@@@tp@stit#1#2#3{\put[0pt,.25\baselineskip]{%
  \hbox{\postitc@lor{\fontfamily{cmr}\fontshape{n}\fontseries{m}\selectfont%
  \tiny\setbox5\hbox{\vbox to #2{\hsize#1\leftskip\z@\raggedright
  \parindent\z@{\ns#3\par}\vss}}%
  \hbox to 0pt{\hss\tiny$\blacktriangledown$\kern-.085em}%
  \raise2.25ex\hbox to 0pt{\hss\framebox{\box5}}}}}}

\def\postitc@lor{\color{postitcolor}}

\def\showpostittrue{\let\postit\postit@}
\def\showpostitfalse{\let\postit\hid@@@}

\long\def\hid@@#1{\ignorespaces}
\def\hid@@@{\@ifnextchar[{\hid@@@@}{\hid@@}}
\long\def\hid@@@@[#1]{\hid@@}

\makeatother

\papersize{216truemm}{279truemm}
\margins{33truemm}{21truemm}
\advance\voffset-3truemm
\advance\hoffset3truemm
\advance\footskip0truemm

\headline{\hfill}
\footline{\small\hfill--\kern1ex\thepage\kern1ex--\hfill}

\sectionbeforeskip{1.5\bigskipamount}
\sectionstyle{\centering\normalsize\sc}
\sectionafterskip{\bigskipamount}
\sectionafterindenttrue

\subsectionbeforeskip{\bigskipamount}
\subsectionstyle{\normalsize\bf}
\subsectionafterskip{.5\bigskipamount}
\subsectionafterindenttrue

\paragraphbeforeskip{\medskipamount}
\paragraphstyle{\ft{cmbxsl10}\boldmath}
\paragraphafterskip{1.25ex}
\paragraphindenttrue
\paragraphdot{.}

\abovedisplayskip\smallskipamount
\belowdisplayskip\smallskipamount
\abovedisplayshortskip\smallskipamount
\belowdisplayshortskip\smallskipamount


\textfloatsep\floatsep

\proofingfalse
\autolabelfalse

\showlabelsfalse
\showcitationstrue
\showreferencestrue

\showfigurestrue
\showdiagramstrue

\newtheorem{stat}{\statname}  \unnumbered{stat}

\newtheorem{nstat}{\nstatname}[section]

\newtheorem[{\ns}{}]{definition}[nstat]{Definition}
\newtheorem{lemma}[nstat]{Lemma}
\newtheorem{proposition}[nstat]{Proposition}
\newtheorem{theorem}[nstat]{Theorem}
\newtheorem{corollary}[nstat]{Corollary}

\newtheorem[{\ns}{}]{exercise}[nstat]{Exercise}
\newtheorem[{\ns}{}]{example}[nstat]{Example}
\newtheorem[{\ns}{}]{remark}[nstat]{Remark}

\let\ns\normalshape

\captionstyle{}
\captionintrostyle{}

\showfigurestrue

\sectionbeforeskip{1.5\bigskipamount}
\sectionstyle{\centering\normalsize\bf}
\sectionafterskip{\bigskipamount}
\sectionafterindenttrue

\subsectionbeforeskip{\bigskipamount}
\subsectionstyle{\normalsize\bf}
\subsectionafterskip{.5\bigskipamount}
\subsectionafterindenttrue

\paragraphbeforeskip{\bigskipamount}
\paragraphstyle{\centering\normalsize\sl}
\paragraphafterskip{.5\bigskipamount}
\paragraphafternewlinetrue
\paragraphafterindenttrue
\paragraphdot{}

\statementindenttrue
\statementintrostyle{\sc}
\renewtheorem[{\ns}{}]{remark}[nstat]{Remark}
\renewtheorem[{\ns}{}]{definition}[nstat]{Definition}

\captionstyle{\small}
\captionintrostyle{\sc}

\def\(#1\){$(${\sl #1}\/$)$}

\def\calstyle#1{{\text{\ft{eusm10}#1}}}

\newcommand{\Int}{\mathop{\mathrm{Int}}\nolimits} 
\newcommand{\Bd}{\partial} 

\renewcommand{\:}{\,{:}\;}

\newcommand{\Z}{{\mathbb{Z}}}
\newcommand{\spinc}{\mathop{\mathrm{spin}^{\mathrm{c}}}\nolimits}
\newcommand{\Spinc}{\mathop{\mathrm{Spin}^{\mathrm{c}}}\nolimits}

\newcommand{\rot}{\mathop{\mathrm{rot}}\nolimits}

\newcommand{\M}{{\calstyle{M}}}
\newcommand{\F}{{\mathcal{F}}}
\newcommand{\Sc}{{\mathcal{S}}^{\mathrm{c}}}
\newcommand{\s}{{\mathfrak{s}}}

\let\theta\textheta

\def\emph#1{{\sl #1}\/}

\font\ftt cmtt10 at 11pt
\font\fsc cmcsc10 at 12pt
\font\fsl cmsl12 at 12pt

\newcommand{\leftrightwavearrow}{\mathrel{\mkern2mu\put[0pt,0.8pt]{\epsfbox{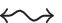}}\mkern24.5mu}}

\paragraphbeforeskip{\bigskipamount}
\paragraphstyle{\centering\normalsize\sl}
\paragraphafterskip{.5\bigskipamount}
\paragraphafternewlinetrue
\paragraphafterindenttrue
\paragraphdot{}

\statementindenttrue
\statementintrostyle{\sc}

\renewtheorem[{\sl}{}]{theorem}[nstat]{Theorem}
\renewtheorem[{\sl}{}]{corollary}[theorem]{Corollary}
\renewtheorem[{\sl}{}]{question}[theorem]{Question}
\renewtheorem[{\sl}{}]{problem}[theorem]{Problem}
\renewtheorem[{\sl}{}]{lemma}[theorem]{Lemma}
\renewtheorem[{\ns}{}]{definition}[theorem]{Definition}
\renewtheorem[{\ns}{}]{remark}[theorem]{Remark}

\begin{document}

\title{\large\bf SPECIAL MOVES FOR OPEN BOOK\\ DECOMPOSITIONS OF 3-MANIFOLDS}
\author{
\fsc Riccardo Piergallini\\
\fsl Scuola di Scienze e Tecnologie\\[-3pt]
\fsl Universit\`a di Camerino -- Italy\\
\ftt riccardo.piergallini@unicam.it
\and
\fsc Daniele Zuddas\\
\fsl Lehrstuhl Mathematik VIII\\[-3pt]
\fsl Mathematisches Institut\\[-3pt]
\fsl Universit\"{a}t Bayreuth -- Germany\\
\ftt zuddas@uni-bayreuth.de}
\date{}

\vglue0pt
\maketitle
\vskip0pt

\begin{abstract}
\baselineskip13.5pt

\vglue12pt
\noindent
We provide a complete set of two moves that suffice to relate any two open book decompositions of a given 3-manifold. One of the moves is the usual plumbing with a positive or negative Hopf band, while the other one is a special local version of Harer's twisting, which is presented in two different (but stably equivalent) forms. Our approach relies on 4-dimensional Lefschetz fibrations, and on 3-dimensional contact topology, via the Giroux-Goodman stable equivalence theorem for open book decompositions representing homologous contact structures.

\smallskip

\medskip\smallskip\noindent
{\sl Key words and phrases}\/: open book decomposition, fibered knot, fibered link,\break 3-manifold, Lefschetz fibration.

\medskip\noindent
{\sl 2010 Mathematics Subject Classification}\/: 57N12, 55R55, 57R17.
\end{abstract}

\section*{Introduction}

An \emph{open book decomposition} of a closed, connected, oriented 3-manifold $M$ consists of a smooth family $\F = \{F_\theta\}_{\theta\in S^1}$ of compact, connected surfaces $F_\theta \subset M$, called the \emph{pages} of the open book, having the same boundary $\Bd F_\theta = L$, which is a smooth link $L \subset M$ with $b \geq 1$ components called the \emph{binding} of the open book. The interiors of the pages are required to be pairwise disjoint, so that the map $\phi \: M - L \to S^1$ defined by $\phi(\Int F_\theta) = \theta$ is a locally trivial fiber bundle over the circle. Therefore, the pages $F_\theta$ are canonically co-oriented and diffeomorphic to each other, and their genus $g \geq 0$ is called the \emph{genus} of the decomposition.

A classical result of Alexander \cite{A1923} guarantees that every closed, connected, oriented 3-manifold $M$ admits an open book decomposition. Further, Myers \cite{My78} and Gonz\'alez-Acu\~na \cite{GA74} independently proved that the binding can be assumed to be connected. On the other hand, the pages can be assumed to be planar, that is of genus $g = 0$, but without any bound on the number of the binding components (see Rolfsen \cite[Chapter 10K]{Ro90}).

Moreover, Harer \cite{H82} showed how to relate any two open book decompositions of the same 3-manifold, in terms of ambient isotopy and certain operations: {\sl stabilization} and \emph{twisting}. More recently, Giroux and Goodman \cite{GG06} proved that twisting is not needed, in other words the two decompositions are \emph{stably equivalent}, if and only if the oriented plane fields given by the associated contact structures are homologous, that is they are homotopic in the complement of a point. In particular, all the open book decompositions of a given integral homology 3-sphere are stably equivalent.

Every open book decomposition of a 3-manifold $M$ as above, can be realized as the boundary of a (possibly achiral) \emph{Lefschetz fibration} $f \: W \to B^2$, whose regular fiber is a connected surface of genus $g$ with $b$ boundary components, and whose total space $W$ turns out to be a \emph{4-dimensional 2-handlebody}, that is a handlebody with handles of indices $\leq 2$ (see Section \ref{lf/sec} for the definition). Namely, $M = \Bd W$ and $F_\theta = f^{-1}(R_\theta) \cap \Bd W$, with $R_\theta \subset B^2$ the radius of the disk ending at $\theta \in S^1$ (hence, $L = f^{-1}(0) \cap \Bd W$ and $\phi(x) = f(x)/\|f(x)\|$ for every $x \in M - L$). 

In \cite{APZ13}, Apostolakis and we proved that two Lefschetz fibrations over the disk have total spaces that are equivalent up to \emph{2-deformation} (handle slidings and births/deaths of pairs of complementary handles of index $\leq 2$), if and only if they are equivalent up to \emph{Hopf stabilizations} $S_\pm$ and certain moves $T$ and $U$ (see Theorem\break \ref{lf/thm}). Moreover, with two extra moves $P$ and $Q$ that realize, respectively, the connected sum with $CP^2$ and the surgery operation of 1/2-handles trading, we can relate any two Lefschetz fibrations whose total spaces have diffeomorphic boundaries.

This provides an alternative approach to the above mentioned Harer's result. In fact, the restrictions of our moves to the boundary $\Bd S_\pm\,,\,\Bd T\,,\,\Bd U$ and $\Bd P$, allow us to relate any two open book decompositions of diffeomorphic 3-manifolds (see Theorem \ref{ob/thm}). The move $\Bd Q$ is not needed, being realizable by ambient isotopy. On the other hand, $\Bd S_\pm$ coincides with the stabilization of open books, while the other boundary moves $\Bd T$ and $\Bd U$ are, up to stabilizations, special instances of Harer's twisting (see Section \ref{rem/sec}).

In this paper, we improve such approach in different directions. First of all, we show that $\Bd P$ can be generated by stabilization of open books (Proposition \ref{P/thm}). Then, based on the above mentioned Giroux-Goodman's result \cite{GG06} and on the effect of move $U$ on the Euler class of a Lefschetz fibration, we prove that, besides stabilization, either a very special case of move $\Bd U$ or a very special case of move $\Bd T$ suffices to relate open book decompositions of diffeomorphic 3-manifolds (Theorems \ref{SU/thm}\break and \ref{ST/thm}). Moreover, for open book decompositions of a given 3-manifold $M$, the moves can be realized as embedded in $M$ and the resulting equivalence of open books can be thought up to ambient isotopy in $M$, not just up to diffeomorphism.

\section{Lefschetz fibrations over the disk\label{lf/sec}}

Given a smooth, compact, connected, oriented 4-manifold $W$ and a compact, connected, oriented surface $S$, both with (possibly empty) boundary, a {\sl Lefschetz\break fibration} $f \: W \to S$ is a smooth map such that: (1) $f$ is regular on the complement of a finite \emph{critical set} $C_f \subset \Int W$ and it is injective on $C_f$; (2) the restriction of $f$ over the complement of the set $f(C_f) \subset \Int S$ of critical values, namely $f_| \: W - f^{-1}(f(C_f)) \to S - f(C_f)$, is a locally trivial bundle with fiber a surface $F$, called the {\sl regular fiber} of $f$; (3) at every critical point, $f$ is locally equivalent, up to orientation-preserving diffeomorphisms, to one of the complex maps $h_\pm \: C^2 \to C$ given by $h_+(z_1, z_2) = z_1 z_2$ and $h_-(z_1,z_2) = z_1 \bar z_2$. It follows that the associated bundle $f_|$ is orientable, and the regular fiber $F$ has a canonical orientation induced by that of $W$ and of $S$. 

A critical point $x \in C_f \subset W$, as well as the corresponding critical value $f(x) \in S$, is said to be \emph{positive} (resp. \emph{negative}) depending on the actual local
model $h_+$ (resp. $h_-$) around it. Then, we say that $f$ is a \emph{positive} Lefschetz fibrations if all of its critical points (and values) are positive. Otherwise, $f$ is said to be \emph{achiral}.

Two Lefschetz fibrations $f \: W \to S$ and $f' \: W' \to S'$ are said to be \emph{equivalent} if there exist orientation-preserving diffeomorphisms $h \: W \to W'$ 
and $k \: S \to S'$ such that $f' \circ h = k \circ f$. Moreover, if $W' = W$, $S' = S$, and both $h \: W \to W$ and $k \: S \to S$ are isotopic to the identity, then $f$ and $f'$ are said to be \emph{isotopic}.

\medskip

\emph{Throughout this paper, we only consider (possibly achiral) Lefschetz fibrations over the disk $B^2$ whose regular fiber $F$ has non-empty boundary. In this case, $F$ is always connected.}

\medskip

Let $f \: W \to B^2$ be a Lefschetz fibration with bounded regular fiber and set of critical values $f(C_f) = \{a_1,\dots,a_n\} \subset \Int B^2$. Once a base point $\ast \in S^1$ and an identification $F_\ast = f^{-1}(\ast) \cong F_{g,b}$ are fixed, with $F_{g,b}$ the standard oriented surface of genus $g \geq 0$ with $b \geq 1$ boundary components, to $f$ is associated a \emph{monodromy representation} $$\omega_f \: \pi_1(B^2 - \{a_1, \dots, a_n\}, \,\ast\,) \to \M_{g,b}\,,$$ where $\M_{g,b}$ is the mapping class group of $F_{g,b}$. Property (3) in the definition of Lefschetz fibrations implies that the monodromy of any positive meridian around a critical value $a_i$ is a Dehn twist, which is positive or negative according to the sign of $a_i$.

If also a {\sl Hurwitz system} for $\{a_1, \dots, a_n\}$ is given, then a specific meridian $\mu_i$ around each $a_i$ turns out to be singled out, in such a way that $\{\mu_1, \dots, \mu_n\}$ is an ordered set of free generators for $\pi_1(B^2 - \{a_1, \dots, a_n\}, \,\ast\,)$ whose product $\mu_1 \cdots \mu_n$ is the positive generator of $\pi_1(S^1,\,\ast\,)$. Then, the monodromy $\omega_f$ can be encoded by the {\sl monodromy sequence} $(\gamma_1 = \omega_f(\mu_1), \dots, \gamma_n = \omega_f(\mu_n))$, where $\gamma_i \in \M_{g,b}$ is (the isotopy class of) a positive or negative, depending on the sign of $a_i$, Dehn twist along a simple closed curve $c_i \subset F_* \cong F_{g,b}$, called the \emph{vanishing cycle} over the critical value $a_i$. It follows that the singular fiber $f^{-1}(a_i)$ is homeomorphic to the quotient space $F_*/c_i$, where the curve $c_i \subset F_*$ is shrunk to a point.

On the other hand, any abstract sequence $(\gamma_1, \dots, \gamma_n)$ of (classes of) positive or negative Dehn twists in $\M_{g,b}$ is the monodromy sequence of a Lefschetz fibration $f \: W \to B^2$ uniquely determined up to equivalence.

It can be proved that two monodromy sequences determine equivalent Lefschetz fibrations if and only if they can be related by a finite sequence of global conjugations and Hurwitz moves, namely transformations of the form
$$
(\gamma_1, \dots, \gamma_i, \gamma_{i+1}, \dots, \gamma_n) \leftrightwavearrow (\gamma_1, \dots, \gamma_i \gamma_{i+1} \gamma_i^{-1}, \gamma_i, \dots, \gamma_n)\,,
$$
which give, in terms of the standard generators, the action of the braid group $B_n = \M(B^2, \{a_1, \dots,\allowbreak a_n\})$ on the set of the Hurwitz systems, composed with the monodromy of the Lefschetz fibration.
\medskip

\emph{In the following, we always assume that Lefschetz fibrations are {\sl relatively minimal}, namely all the vanishing cycles are homotopically non-trivial curves in $F$.} 

\medskip

A more restrictive property of a Lefschetz fibration is allowability, which we will explicitly require when needed. We recall that a Lefschetz fibration is said to be \emph{allowable} if none of its vanishing cycles is null-homologous in the regular fiber. In particular, relatively minimal (resp. allowable) implies that no singular fiber contains a 2-sphere (resp. a closed surface).

\medskip

We conclude this section with the statement of the equivalence theorem for Lefschetz fibrations over $B^2$ that was proved in \cite[Theorem A]{APZ13}. In order to state the theorem, we recall the 4-dimensional instance of the Kas \cite{K80} theorem that a Lefschetz fibration $f \: W \to B^2$ naturally induces a 4-dimensional 2-handlebody structure $H_f$ on the total space $W$. This consists of a single 0-handle, $2g + b - 1$ 1-handles and $n$ 2-handles, where $g \geq 0$ is the genus of the regular fiber and $b \geq 1$ is the number of its boundary components, while $n$ is the number of critical points. The 0-handle and the 1-handles form a tubular neighborhood $N \cong F_\ast \times B^2$ of the regular fiber $F_\ast = f^{-1}(\ast)$, and the 2-handles $H^2_1, \dots, H^2_n$ are attached to $N$ along parallel copies of the vanishing cycles $c_i$ in different fibers, respectively, $F_{*_1}, \dots, F_{*_n} \subset \Bd N$, cyclically ordered by the Hurwitz system, and with framings $-\epsilon_i$ (with respect to the fiber), where $\epsilon_i = \pm 1$ is the sign of the corresponding critical point $a_i$, for $i = 1,\dots, n$ (see also \cite{APZ13} for more details).

\begin{theorem}[\cite{APZ13}]\label{lf/thm}
Two Lefschetz fibrations $f \: W \to B^2$ and $f' \: W' \to B^2$ represent 2-deformation equivalent 4-dimensional 2-handlebodies $H_f$ and $H_{f'}$ if and only if they are related by equivalence, stabilizations $S_\pm$ and moves $T$ and $U$ (and their inverses), described in terms of modifications of the regular fiber and of the monodromy sequence in Figures \ref{stab/fig}, \ref{Tmove/fig} and \ref{Umove/fig}, respectively.
\end{theorem}

In the stabilization moves $S_\pm$ shown in Figure \ref{stab/fig}, we add a band to the fiber and a new positive or negative Dehn twist in the monodromy sequence, along a curve that crosses once the new band. The inverse moves are usually called {\sl destabilizations}.

In move $T$, a band as in Figure \ref{Tmove/fig} is removed and glued back differently, along with the two Dehn twists $\gamma_i$ and $\gamma_{i+1}$. The band involved in the move does not meet any other vanishing cycle. Moreover, the signs $\epsilon_j = \pm1$ of the involved Dehn twists, for $j = i,\, i+1$, can be taken arbitrary, but restricting to the case $\epsilon_i = - \epsilon_{i+1}$ suffices.

In move $U$ shown in Figure \ref{Umove/fig}, we consider a disk $D \subset F_*$ disjoint from the vanishing cycles. We remove two small disks in $D$ and add the two boundary parallel and opposite Dehn twists $\gamma_+$ and $\gamma_-$ to the monodromy sequence. The band between them, in the right part of the same figure, does not meet other vanishing cycle.

\begin{Figure}[htb]{stab/fig}
\fig{}{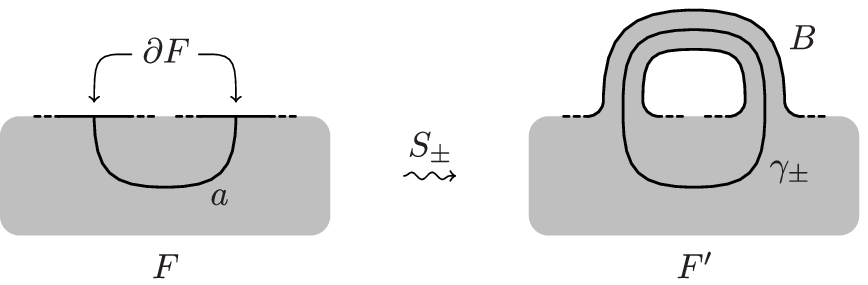}
    {}{Stabilization moves for Lefschetz fibrations. Here, $a$ is an arbitrary proper simple arc in $F$, and $\gamma_\pm$ is the positive or negative Dehn twist along the depicted simple closed curve in $F'$ that extends $a$ across the band $B$, inserted anywhere in the monodromy sequence of the left side.}
\vskip-3pt
\end{Figure}

\begin{Figure}[htb]{Tmove/fig}
\fig{}[12pt]
{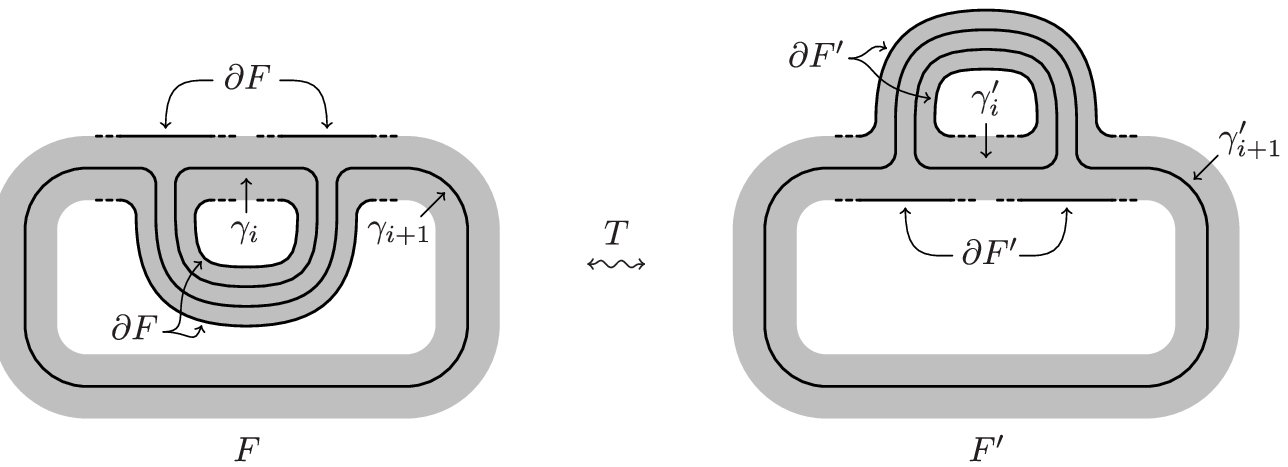}
    {}{The $T$ move for Lefschetz fibrations. The consecutive Dehn twists $\gamma_i$ and $\gamma_{i+1}$ in the monodromy sequence of the left side have opposite sign, and the same holds for the consecutive Dehn twists $\gamma'_i$ and $\gamma'_{i+1}$ in the monodromy sequence of the right side.}
\vskip-3pt
\end{Figure}

\begin{Figure}[htb]{Umove/fig}
\fig{}{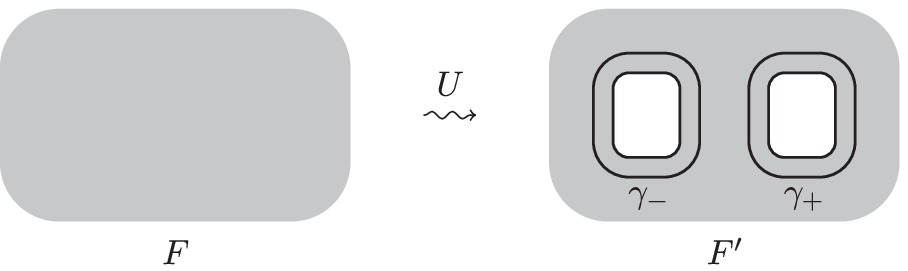}
    {}{The $U$ move for Lefschetz fibrations. The two holes on the right side are contiguous, that is no other vanishing cycle pass between them, and the Dehn twists $\gamma_\pm$ have opposite sign and can be inserted anywhere in the monodromy sequence of the left side.}
\vskip-3pt
\end{Figure}

When dealing with open book decompositions of 3-manifolds thought as boundaries of Lefschetz fibrations, we also need the move $P$ described in the following Figure \ref{Pmove/fig}, and the move $Q$, which inserts (or removes) two consecutive parallel and opposite Dehn twists $\gamma$ and $\gamma^{-1}$ in the monodromy sequence of a Lefschetz fibration.

\begin{Figure}[htb]{Pmove/fig}
\fig{}{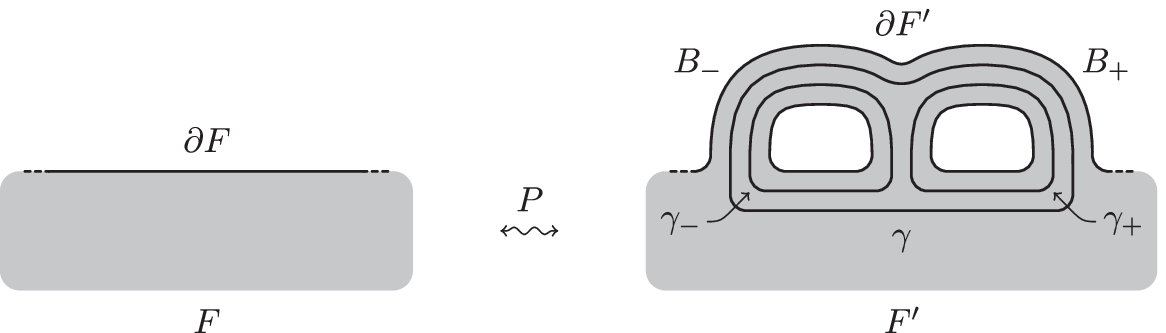}
    {}{The $P$ move for Lefschetz fibrations. The Dehn twists $\gamma_\pm$ have opposite sign, while the sign of the Dehn twist $\gamma$ is arbitrary, and such twists can be inserted anywhere in the monodromy sequence of the left side.}
\vskip-3pt
\end{Figure}

By considering the Kas handlebody decomposition associated to a Lefschetz fibration, it can be proved, by means of Kirby calculus, that the effect of the moves $S_\pm$, $T$ and $U$ on the total 4-manifold consists of handlebody 2-deformation (hence, diffeomorphism), while move $P$ gives connected sum with $CP^2$ and move $Q$ can be suitably used to transform dotted circles in a Kirby diagram into 0-framed circles, up to 2-deformation, see \cite{APZ13} for the details.

\begin{remark}\label{U/rmk}
Among our moves, $U$ and $Q$ are the only ones that make sense even when the fiber of a Lefschetz fibration has no boundary. In this case, the move $U$ still preserves the total 4-manifold up to diffeomorphisms (actually, up to 2-deformation equivalence of 2-handlebodies), while $Q$ preserves the boundary as a fibered 3-manifold.
\end{remark}

\section{Open book decompositions of 3-manifolds\label{ob/sec}}

First of all, we recall that two open book decompositions $\F = \{F_\theta\}_{\theta \in S^1}$ on $M$ and $\F' = \{F'_\theta\}_{\theta \in S^1}$ on $M'$, as defined in the Introduction, are said to be \emph{equivalent} if there exist orientation-preserving diffeomorphisms $h \: M \to M'$ and $k \: S^1 \to S^1$ such that $h(F_\theta) = F'_{k(\theta)}$ for every $\theta \in S^1$, or equivalently $\phi' \circ h = k \circ \phi$, with $\phi \: M - L \to S^1$ and $\phi' \: M' - L' \to S^1$ the associated fiber bundles. Moreover, if $M' = M$ and $h \: M \to M$ is isotopic to the identity, then the two open book decompositions $\F$ and $\F'$ are said to be \emph{isotopic}.

In particular, $\F$ and $\F'$ are equivalent (resp. isotopic) if they are boundaries of equivalent (resp. isotopic) Lefschetz fibrations as described in the Introduction.

Similarly to Lefschetz fibrations, also open book decompositions can be described in terms of their monodromy, defined as follows. For any open book decomposition $\F = \{F_\theta\}_{\theta \in S^1}$ on $M$ with genus $g \geq 0$ and $b \geq 1$ binding components, the associated fiber bundle $\phi \: M - L \to S^1$ can be completed in a canonical way to an $F_{g,b}$-bundle $\widehat\phi \: \widehat M\, \to S^1$, which is trivial on the boundary and whose fiber $\widehat\phi^{-1}(\theta)$ is canonically identified with the page $F_\theta$ for every $\theta \in S^1$. Then, once a base point $\ast \in S^1$ and an identification $F_\ast \cong F_{g,b}$ are chosen, we define the \emph{monodromy} $\omega_{\F} \in \M_{g,b}$ of the open book decomposition $\F$ to be the monodromy of the $F_{g,b}$-bundle $\widehat\phi$. 

The monodromy $\omega_\F$ completely determines the open book decomposition $\F$ up to equivalence. Moreover, if also the inclusion $F_\ast \subset M$ is given, then $\F$ turns out to be determined up to isotopy. 

On the other hand, any $\omega \in \M_{g,b}$ is the monodromy of an open book decomposition $\F_\omega$ of a smooth closed, connected, oriented 3-manifold $M_\omega$ (uniquely determined up to equivalence). In fact, if $f \: F_{g,b} \to F_{g,b}$ is any map in the class $\omega$ that is the identity on the boundary, and $\phi \: F_{g,b} \times R \to F_{g,b} \times R$ is the map defined by $\phi(x,t) = (f(x),t-2\pi)$, we can consider the mapping torus $T(\phi) = (F_{g,b} \times R)/\langle\phi\rangle$, as the quotient of $F_{g,b} \times R$ with respect to the action of the diffeomorphisms group generated by $\phi$. The boundary $\Bd T(\phi)$ can be canonically identified with $\Bd F_{g,b} \times S^1$, so we can define $M_\omega$ as the quotient of $T(\phi)$ obtained by collapsing the $S^1$ fibers of this product.
Then, the open book decomposition $\F_\omega$ has pages $F_\theta = \pi(F_{g,b} \times \{\theta\})$ where $\pi \: F_{g,b} \times R \to M_\omega$ is the canonical projection resulting from the quotients.

It immediately follows from this construction, that the open book decomposition $\F_{\omega^{-1}}$ is equivalent to $\F_\omega$ with reversed orientation, and that $\F_\omega$ is equivalent to $\F_{\psi \omega \psi^{-1}}$ for all $\psi\in \M_{g,b}$.

If an open book decomposition $\F$ is the boundary of a Lefschetz fibration $f \: W \to B^2$ with monodromy sequence $(\gamma_1, \dots, \gamma_n)$, then the monodromy of $\F$ is given by the product $\omega_\F = \gamma_1 \cdots \gamma_n$. In light of the above discussion, any Lefschetz fibration $f \: W \to B^2$ having $\F$ as the boundary
is determined in this way by a factorization of $\omega_\F$ (or a conjugate of it) as a product of Dehn twists. Conversely, any factorization of $\omega_\F$ as a product of Dehn twists determines such a Lefschetz fibration $f$. In particular, we can choose all the Dehn twists in the factorization to be not null-homologous, and hence the Lefschetz fibration $f$ to be allowable.

The following equivalence theorem was proved in \cite[Theorem C]{APZ13}, based on the representation of open book decompositions of 3-manifolds as boundaries of allowable Lefschetz fibrations over the disk.

\begin{theorem}\label{ob/thm}
Any two open book decompositions of a closed, connected, oriented 3-manifold $M$ are related by equivalence and the moves $\partial S_\pm$, $\partial T$ and $\partial P$ (and their inverses), which are the restriction to the boundary of the moves $S_\pm\,,\,T$ and $P$ described in Figures 1, 2 and 4, respectively.
\end{theorem}

Notice that move $\Bd U$ is not mentioned in the theorem. Actually, it does not change the manifold $M$ supporting the open book decomposition, being the restriction to the boundary of move $U$, but nevertheless it is not needed here. The reason is that in Theorem \ref{lf/thm} it plays an auxiliary role, being used only to make Lefschetz fibrations allowable.

\section{The equivalence results for open books}

First of all, we show that the move $\Bd P$ can be removed from Theorem \ref{ob/thm}

\begin{Figure}[b]{PfromS/fig}
\vskip3pt
\fig{}{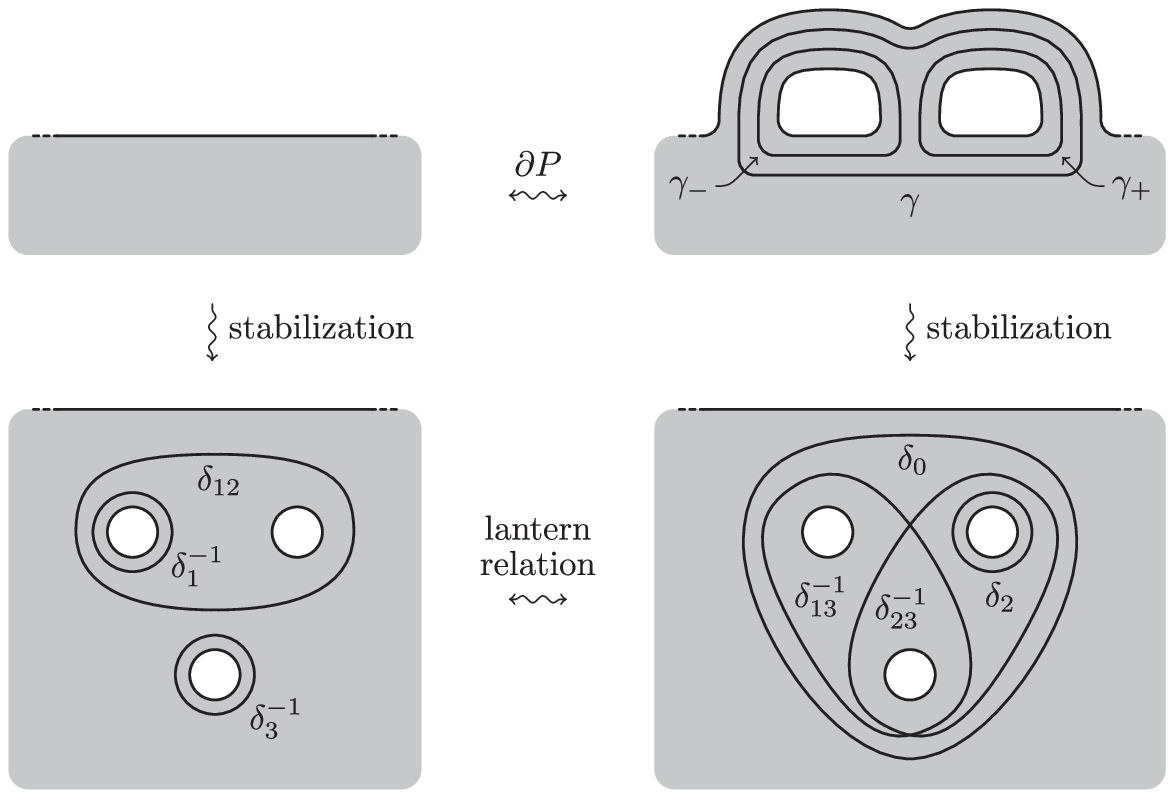}
    {}{The $\Bd P$ move as a consequence of $\Bd S_\pm$.}
\end{Figure}

\begin{proposition}\label{P/thm}
On any closed, connected, oriented 3-manifold $M$, the move $\Bd P$ for open book decompositions on $M$ is a consequence of $\Bd S_\pm$.
\end{proposition}

\begin{proof}
In Figure \ref{PfromS/fig} we consider the move $P$ in the case where $\gamma$ is a positive twist, the case of a negative twist being similar with all the twists inverted. Here, we perform a triple stabilization of the open book on the left side and a single stabilization of that on the right side. This last stabilization creates the new twist $\delta_{23}^{-1}$, while the other ones $\delta_0\,,\delta_{13}^{-1}$ and $\delta_2$ are the preexisting twists $\gamma\,,\gamma_-$ and $\gamma_+$, respectively. After that the move reduces to the relation $\delta_1^{-1}\delta_3^{-1}\delta_{12} = \delta_0\delta_2\delta_{23}^{-1} \delta_{13}^{-1}$, which is equivalent to the lantern relation $\delta_{12}\delta_{13} \delta_{23} = \delta_0\delta_1\delta_2\delta_3$ in $\M_{g,b}$, up to commutation of disjoint twists.
\end{proof}

Now, in order to prove our main results, we need to recall some facts concerning $\spinc$ structures on 4-dimensional 2-handlebodies and 3-manifolds and their relations with Lefschetz fibrations and open book decompositions, respectively. The reader can see the book of Ozbagci and Stipsicz \cite{OS2004} as a general reference.

For a smooth oriented (possibly bounded) manifold $M$, we denote by $\Sc(M)$ the set of $\spinc$ structures on $M$. By Gompf \cite{G97}, if $\dim M \geq 3$ then $\Sc(M)$ can be identified with the set of homotopy classes of almost complex structures over the 2-skeleton of (any fixed cellular decomposition of) $M$ that can be extended over the 3-skeleton, after stabilizing $T M$ with a trivial real line bundle if $\dim M$ is odd.

The set $\Sc(M)$ is non-empty if and only if the second Stiefel-Whitney class $w_2(M)$ is in the image of the coefficient homomorphism $H^2(M) \to H^2(M; \Z_2)$. In\break fact, the  Chern class $c_1(\s)$ of any $\s \in \Sc(M)$ is mapped to $w_2(M)$ by such homomorphism. In this case, the group $H^2(M)$ acts freely and transitively on $\Sc(M)$, namely $\Sc(M)$ is a {\sl torsor} over $H^2(M)$. The action of $\lambda \in H^2(M)$ on $\s \in \Sc(M)$ will be denoted by $\s + \lambda \in \Sc(M)$. Then, we have $c_1(\s + \lambda) = c_1(\s) + 2 \lambda$ for every $\s \in \Sc(M)$ and $\lambda \in H^2(M)$. In particular, the map $c_1 \: \Sc(M) \to H^2(M)$ is injective if $H^2(M)$ has no 2-torsion.

Given a (possibly achiral) Lefschetz fibration $f \: W \to B^2$, we can associate to $f$ a $\spinc$ structure $\s_f \in \Sc(W)$, in the following way. Perform a small isotopy making $C_f$ disjoint from the 2-skeleton of $W$. Let $\xi_f \subset TW$ be the oriented distribution on $W - C_f$ tangent to the fibers of $f$. Put a Riemannian metric on $TW$, and let $\xi_f^\perp \subset TW$ be the normal distribution of $\xi_f$ on $W - C_f$. Consider the almost complex structure $J$ on $TW$ over $W - C_f$ that keeps $\xi_f$ and $\xi_f^\perp$ invariant and acts as a $\pi/2$-rotation on both distributions. Then, $\s_f$ is by definition the homotopy class of the restriction of $J$ over the 2-skeleton of $W$, and it is independent on the choice of the small isotopy and of the metric. Since $W$ deformation retracts on its 2-skeleton, $\s_f$ represents a $\spinc$ structure.

Since the normal distribution $\xi_f^\perp \cong f^*{(T B^2)}$ is trivial as a bundle over $W - C_f$ and
the inclusion $i \: W - C_f \to W$ induces an isomorphism $i^* \: H^2(W) \to H^2(W - C_f)$, the Chern class of $\s_f$ can be expressed by $c_1(\s_f) = (i^*)^{-1}(e(\xi_f))$, where $e$ denotes the Euler class. 

Moreover, since we are assuming that the regular fiber $F$ of $f$ has non-empty boundary, the Chern class $c_1(\s_f)$ coincides with the Euler class $e(f)$ considered in \cite{APZ13}, and hence it can be computed as follows (see also Gompf \cite{G98} and Gay and\break Kirby \cite{GK2004}). Let $c_1, \dots, c_n$ be the vanishing cycles of $f$ determined by a Hurwitz system for $f(C_f) = \{a_1, \dots, a_n\} \subset B^2$. By recalling that the 2-handles of the handlebody decomposition $H_f$ of $W$ are attached along parallel copies of such vanishing cycles, we can identify the cellular 2-chain group $C_2(W)$ with the free abelian group generated by the cycles $c_1, \dots, c_n$ with any given orientation, and the cellular 2-cochain group $C^2(W)$ with the free abelian group generated by the dual generators $c_1^*, \dots, c_n^*$. Then, denoting by $\rot(c)$ the rotation number of the positive unit tangent vector along $c$ with respect to any given trivialization of the tangent bundle $TF$, we have that $c_1(\s_f)$ is the cohomology class of
$$\sum_{i=1}^n \rot(c_i)\, c^*_i \in C^2(W).$$

In light of this formula, it is easy to check that the moves $S_\pm$ do not change $c_1(\s_f)$, while the moves $T$ and $U$ change it by adding an even cohomology class.

\begin{lemma}\label{Euler/thm}
Let $f \: W \to B^2$ be a (possibly achiral) Lefschetz fibration, and let $\eta \in H^2(W)$ be any integral lift of $w_2(W)$. Then, there is a finite sequence of moves $U$ that transforms $f$ into an allowable Lefschetz fibration $f' \: W \to B^2$ such that $c_1(\s_{f'}) = \eta$.
\end{lemma}

\begin{proof}
Up to move $U$ we can assume $f$ itself to be allowable. Moreover, since $c_1(\s_f)$ is an integral lifting of $w_2(W)$ as well, the universal coefficient theorem implies that $\eta - c_1(\s_f)$ is an even class. Then, by keeping the above notation, it is enough to show how move $U$ can be used to add $\pm 2 [c^*_i]$ to $c_1(\s_f)$, for all $i = 1, \dots, n$. This is shown in Figure \ref{Euler/fig}, where the top and bottom lines represent (parts of) $\Bd F$, and $c_i$ is pushed along the arcs $a$ and $b$, respectively, before applying the move. These arcs are such that $a \cup b$ is a properly embedded arc that meets $c_i$ transversally at one point (the common endpoint of $a$ and $b$), and their existence is guaranteed by the allowability of $f$.

\begin{Figure}[htb]{Euler/fig}
\fig{}{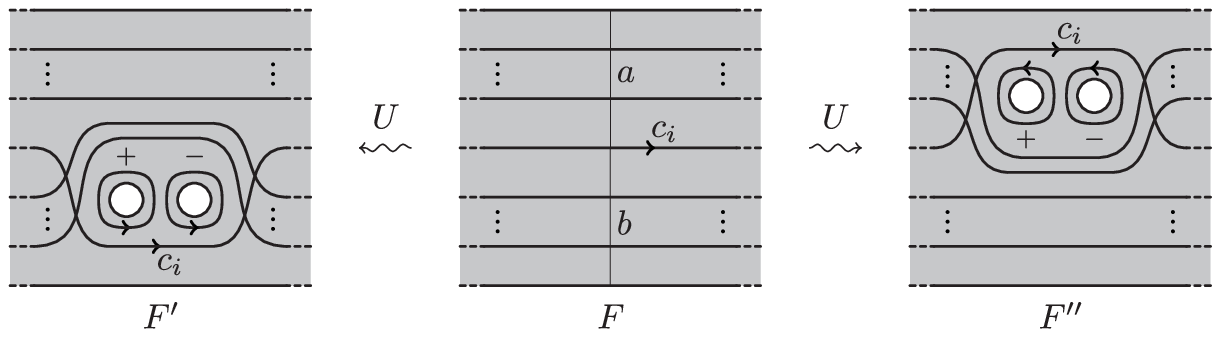}
    {}{How to change the Chern class $c_1(\s_f)$ by using the move $U$.}
\vskip-3pt
\end{Figure}

A simple computation shows that the move $U$ on the left side changes $c_1(\s_f)$ by $-2[c_i^*]$, while that on the right side changes $c_1(\s_f)$ by $+2[c_i^*]$. Indeed, by considering the change of basis in $C_2(W)$, together with the relations in (co)homology, and by taking the trivialization of $T F$ induced by a suitable planar immersion $F \to R^2$, we obtain $[c_+^*] = [c_-^*] = -[c_i^*]$ for the left side, and $[c_+^*] = [c_-^*] = [c_i^*]$ for the right side, where $c_+$ and $c_-$ denote the cycles with sign $+$ and $-$, respectively. Moreover, in both cases, $\rot(c_+) = \rot(c_-) = 1$. The other generators, as well as the corresponding rotation numbers, are not affected by the move.
\end{proof}

\begin{remark}\label{isotopy/rmk}
The identification of the total spaces of the Lefschetz fibrations before and after the move is canonical, up to isotopy, because it is given by 2-handle slidings and 1/2-handles births/deaths, which can be realized as embedded in $W$. This is important for having {\sl isotopy} instead of open book equivalence in next Theorems \ref{SU/thm} and \ref{ST/thm}.
\end{remark}

\begin{lemma}\label{spinc-thm}
Let $f \: W \to B^2$ be a (possibly achiral) Lefschetz fibration such that $H^2(W)$ has no 2-torsion. Then, for every $\s \in \Sc(W)$ there is a finite sequence of moves $U$ that transforms $f$ into an allowable Lefschetz fibration $f'$ such that $\s_{f'} = \s$.
\end{lemma}

\begin{proof}
The assumption on $H^2(W)$ implies that $\s$ in uniquely determined by its Chern class $c_1(\s)$. Thus, the conclusion follows from Lemma \ref{Euler/thm}.
\end{proof}

\medskip

To an open book decomposition $\F$ on a closed connected oriented 3-manifold $M$, we can  associate a $\spinc$ structure $\s_\F \in \Sc(M)$ by the same construction described above to define $\s_f$, with $TM \oplus R$ in place of $TW$ and the plane field $\xi \subset TM$ given by the contact structure associated to $\F$, which is determined by the Thurston-Winkelnkemper construction \cite{TW75}, in place of $\xi_f$. In this case, $\s_\F$ uniquely determines the restriction of the plane field $\xi$ over the 2-skeleton of $M$ up to homotopy. In fact, such restriction is homotopic to the plane field $TM \cap J(TM)$ over the 2-skeleton of $M$, see Ozbagci and Stipsicz \cite[Chapter 6]{OS2004} and Geiges \cite[Chapter 4]{Ge08} for more details. In particular, if the cellular decomposition of $M$ is chosen to have only one 3-cell, then $\s_\F$ determines $\xi$ up to homotopy over the complement of a point.

\medskip
At this point, we can state our equivalence theorems.

\begin{theorem}\label{SU/thm}
Any two open book decompositions of a closed, connected, oriented 3-manifold $M$ can be related by isotopy, stabilizations, and the special case of move $\Bd U$ in Figure \ref{Us/fig} (and their inverses).
\end{theorem}
\begin{Figure}[htb]{Us/fig}
\fig{}{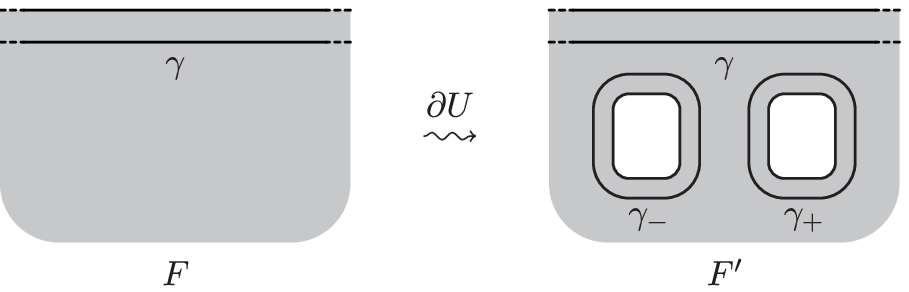}
    {}{Special case of move $\Bd U$. Here, in addition to what we said in the caption of Figure \ref{Umove/fig}, the sign of the Dehn twist $\gamma$ is arbitrary, and, on the right side, no vanishing cycles other than $\gamma$ and $\gamma_\pm$ separate the holes from the drawn part of the boundary.}
\vskip-6pt
\end{Figure}

\begin{proof}
Let $W = H^0 \cup H^2_1 \cup\dots \cup H^2_n$ be the oriented 4-dimensional 2-handlebody without 1-handles determined by an integral surgery presentation of $M \cong \Bd W$. Since $W$ is simply connected, the homomorphism $i^* \: H^2(W) \to H^2(M)$ induced by the inclusion $i \: M \to W$ is surjective, due to the exact cohomology sequence of the pair $(W, M)$. Then, also the induced natural map $i^* \: \Sc(W) \to \Sc(M)$ is surjective,
that is every $\spinc$ structure on $M$ extends over $W$.

Now, assume we are given any two open book decompositions $\F$ and $\F'$ of $M$. Since $H^2(W)$ has no torsion, we can construct two (possibly achiral) allowable Lefschetz fibrations $f, f' \: W \to B^2$, such that the associated $\spinc$ structures $\s_f,\, \s_{f'} \in \Sc(W)$ extend the $\spinc$ structure $\s_\F,\, \s_{\F'} \in \Sc(M)$ associated to $\F$ and $\F'$, respectively. Indeed, $f$ and $f'$ can be obtained by applying Lemma \ref{spinc-thm} to a given Lefschetz fibration $g \: W \to B^2$, whose existence is guaranteed by a result of Harer \cite{Ha79} (see also Etnyre and Fuller \cite{EF06} for a sketch of Harer's original proof). In particular, $f'$ can be obtained from $f$ by a finite sequence of moves $U$, and their inverses.

Let $\Bd f$ and $\Bd f'$ denote the open book decompositions of $M$ given by the boundary restrictions of $f$ and $f'$, respectively. Then, by the observation we made just before the theorem, the equality $\s_{\F} = \s_{\Bd f}$ implies that the contact structures associated to $\F$ and  $\Bd f$ are homotopic over the complement of a point. Hence, by the Giroux and Goodman theorem \cite{GG06}, $\F$ and $\Bd f$ are stably equivalent. Similarly, $\F'$ and $\Bd f'$ are stably equivalent as well. On the other hand, $\Bd f$ and $\Bd f'$ are related by a finite sequence of moves $\Bd U$, and their inverses, all having the special form described in Figure \ref{Us/fig} (cf. Figure \ref{Euler/fig}).
\end{proof}

\begin{theorem}\label{ST/thm}
Any two open book decompositions of a closed, connected, oriented 3-manifold $M$ can be related by isotopy, stabilizations, and the special case of move $\Bd T$ in Figure \ref{Ts/fig} (and their inverses).
\end{theorem}

\begin{Figure}[ht]{Ts/fig}
\vskip-9pt
\fig{}[15pt]
{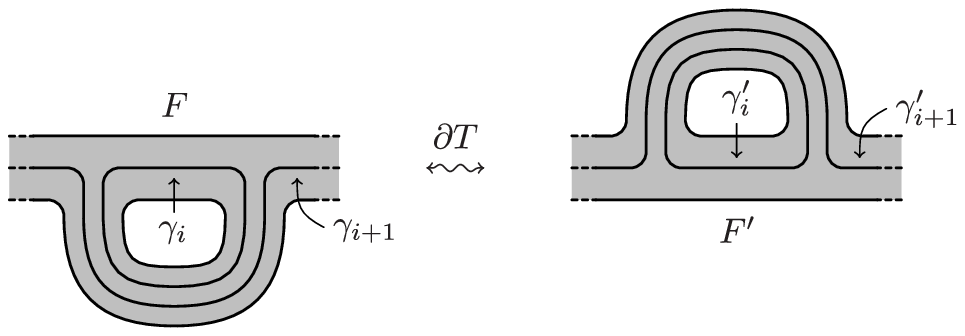}
    {}{Special case of move $\Bd T$. Of course, all the properties required caption of Figure \ref{Tmove/fig} hold here as well.}
\vskip-3pt
\end{Figure}

\begin{Figure}[b]{EulerT/fig}
\fig{}{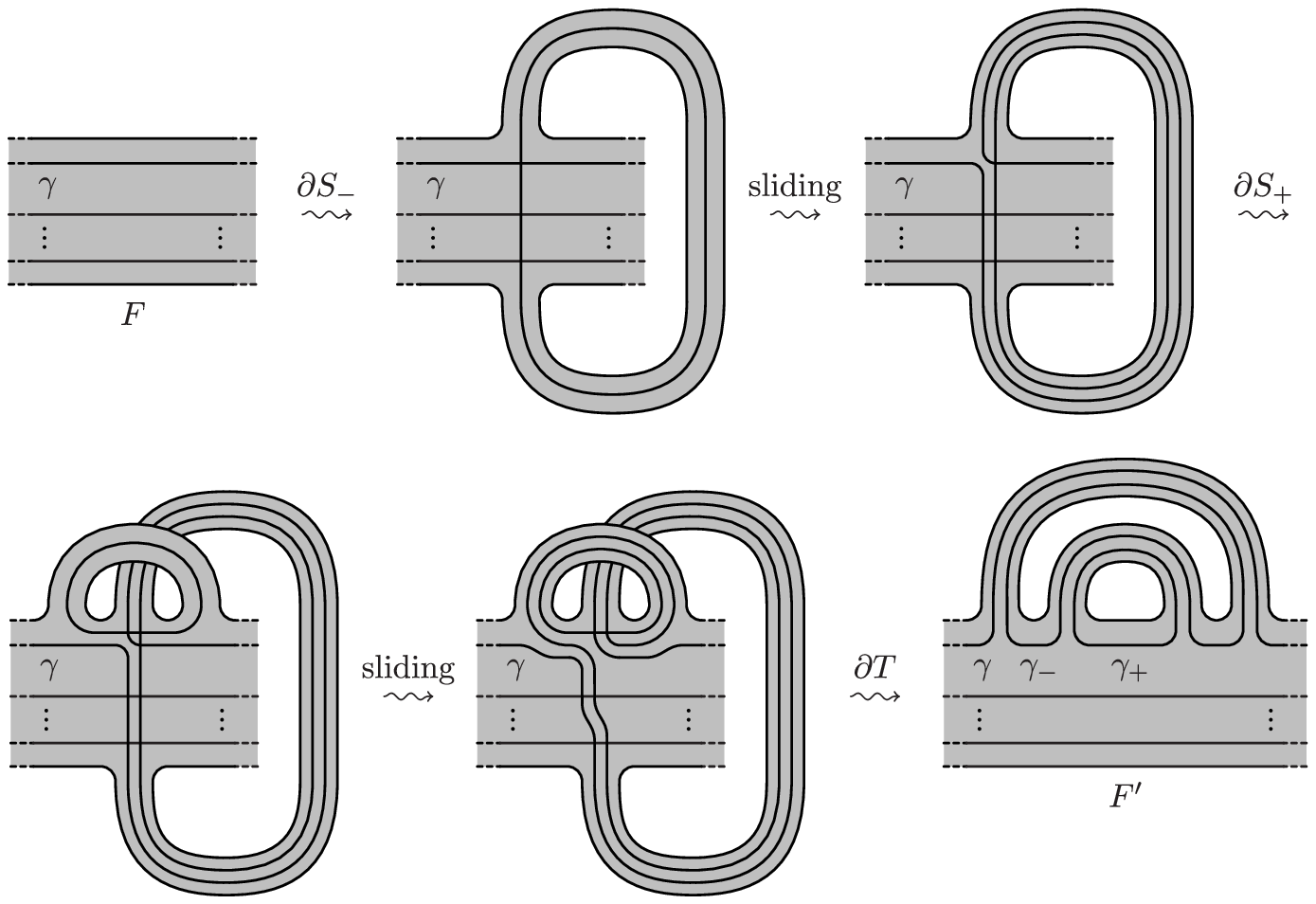}
    {}{The special $\Bd U$ move in terms of $\Bd S_\pm$ and the special $\Bd T$ move.}
\end{Figure}

\begin{proof}
Figure \ref{EulerT/fig} shows how to realize the special $\Bd U$ move in Figure \ref{Us/fig} on the boundary of an allowable Lefschetz fibration, by two stabilizations and a single special $\Bd T$ as in Figure \ref{Ts/fig}. Then, the theorem immediately follows from Theorem \ref{SU/thm} and its proof.
\end{proof}

As a corollary, we have the following.

\begin{corollary}
Let $W$ be a compact, connected, oriented, 4-dimensional 2-handlebody with $H_1(W) = 0$. Then, every open book decomposition of $M = \Bd W$ admits a stabilization that can be extended to a (possibly achiral) allowable Lefschetz fibration $f \: W \to B^2$.
\end{corollary}

\section{Final remarks\label{rem/sec}}
\subparagraph{Harer's twisting} We recall the Harer twisting operation for arbitrary 3-mani\-folds that was described in \cite[Section 4]{H82}. Let $\F$ be an open book decomposition of a 3-manifold $M$ with page $F \cong F_{g,b}$ and with monodromy $\omega_{\F} \in \M_{g,b}$, and consider two pages $F_1,\, F_2 \subset M$ of $\F$. Let $c_i \subset \Int F_i$ be a simple closed curve, and let $\epsilon_i = \pm1$, for $i =1, 2$. Suppose that there exists an annulus $A \cong S^1 \times [0,1]$ embedded in $M$, such that $\Bd A = c_1 \cup c_2$. Fix an orientation on $A$, and consequently on $\Bd A$. Let $c'_i$ be a simple curve obtained by pushing off $c_i$ in $F_i$ with one extra $\epsilon_i$ full twist, oriented accordingly, and let $\epsilon'_i$ be the algebraic intersection of $A$ and $c'_i$. If $\epsilon'_1 = - \epsilon'_2 = \pm1$, then the open book $\F$ can be modified into an open book $\F'$ on $M$ with the same page and monodromy $$\omega_{\F'} = \omega_{\F}\, t_{c_1}^{-\epsilon_1} t_{c_2}^{-\epsilon_2},$$ where $t_c$ denotes the positive Dehn twist along a simple curve $c \subset F$.

\begin{Figure}[b]{Us-Harer/fig}
\vskip6pt
\fig{}{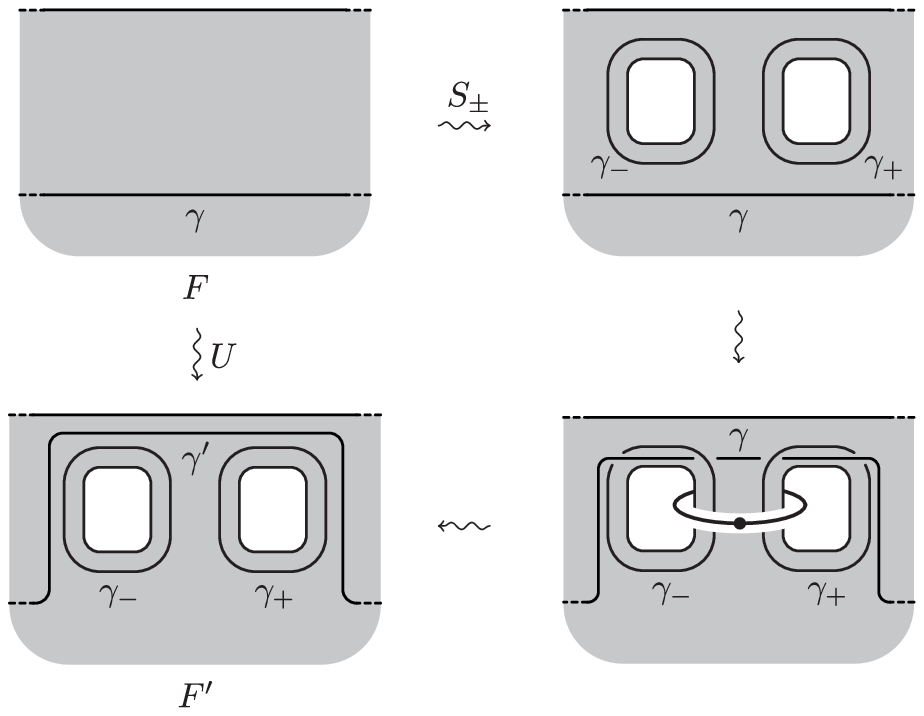}
    {}{The special $U$ move as a Harer twisting}
\end{Figure}

Now, we want to show explicitly how the special case of the $\Bd U$ move in Figure \ref{Us/fig} can be expressed as a Harer twisting and two stabilizations. Indeed, following Figure \ref{Us-Harer/fig}, we start with two opposite stabilizations in the upper part, and then, by using the associated handlebody decomposition, we slide the attaching curve $\gamma$ over $\gamma_+$ and $\gamma_-$. This does not affect the framing, and sliding $\gamma$ over the 1-handle between $\gamma_+$ and $\gamma_-$ gives the handlebody corresponding to the application of the $U$ move. Now, it is enough to observe that the above slidings determine an annulus between $\gamma$ and $\gamma'$ as in the Harer twisting. For the boundary open book, the transformation of the monodromy from the upper right part of Figure \ref{Us-Harer/fig} to the lower left part of it, is multiplication by $\gamma^{-1} \gamma'$. Notice that this gives a 4-dimensional realization of the Harer twisting.

This observation, together with Theorem \ref{SU/thm}, leads to an alternative proof of Theorem 2 in Harer's paper \cite{H82}. 

A similar interpretation can be given for the special $T$ move.

\subparagraph{Effect on the contact structure} Let $\F$ be an open book decomposition of a 3-manifold $M$. Then, the Thurston-Winkelnkemper construction \cite{TW75} gives a contact structure $\xi_\F$ on $M$ compatible with $\F$, which is well-defined up to isotopy (see Giroux \cite{Gi02}).

We discuss the effect of a move applied to $\F$ on the contact structure $\xi_\F$.
The followings are well-known: (1) by the Giroux theorem \cite{Gi02}, two open book decompositions of $M$ are compatible with isotopic contact structures if and only if they can be related by positive stabilizations $S_+$, and isotopy; (2) by Torisu \cite{To2000}, a negative stabilization $S_-$ makes the corresponding contact structure overtwisted.

The $\Bd U$ move depicted in Figure \ref{Umove/fig} makes the corresponding contact structure overtwisted. To see this, consider the curve $c$ in Figure \ref{U-OT/fig}. This curve bounds a disk $D \subset M$ (as it can be easily proved by sliding it as in Figure \ref{Us-Harer/fig}). Moreover, the framing of $c$ with respect to the page and that with respect to $D$ are the same. By the main theorem of Yamamoto \cite{Ya07}, the contact structure is overtwisted (the curve $c$ is a {\sl twisting loop} in Yamamoto paper's terminology).
\begin{Figure}[htb]{U-OT/fig}
\fig{}{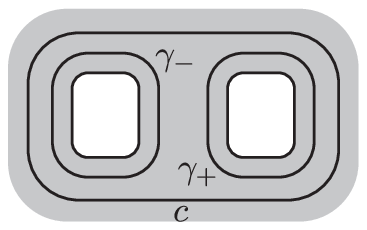}
    {}{Overtwisted disk after a $\Bd U$ move.}
\end{Figure}

By adapting the proof of the main theorem of Ding, Geiges and Stipsicz \cite{DGS05}, one should be able to prove that a special $\Bd U$ move as in Figure \ref{Us/fig}, realized in the stabilized version as the transformation between the upper right and the lower left parts of Figure \ref{Us-Harer/fig}, modifies the contact structure by a {\sl Lutz twist} along the transverse push-off of the Legendrian knot represented by the curve $\gamma$.

A twisting loop can be easily identified in both sides of Figure \ref{Tmove/fig} (as the core of the bigger annulus), concluding that a $\Bd T$ move can be applied only to an overtwisted open book, and it preserves the overtwistedness.

\section*{Acknowledgements}
The second author acknowledges support of the 2013 ERC Advanced Research Grant 340258 TADMICAMT.

The authors are members of the group GNSAGA of Istituto Nazionale di Alta Matematica ``Francesco Severi", Italy.

\end{document}